\documentclass[reqno,12pt]{amsart}
\usepackage{amsmath,amstext, amsthm,
  amssymb,epsfig,color,tikz,mathrsfs,longtable}
\usetikzlibrary{arrows}
\usetikzlibrary{positioning}
\numberwithin{equation}{section}
\newtheorem{thm}{Theorem}[section]
\newtheorem{lem}[thm]{Lemma}

\newtheorem{cor}[thm]{Corollary}

\newcommand{\be}{\begin{equation}}
\newcommand{\ee}{\end{equation}}
\newcommand{\ba}{\begin{array}}
\newcommand{\ea}{\end{array}}

\renewcommand{\em}{\it}

\newcommand{\al}{\alpha}

\newtheorem{rem}[thm]{Remark}
\newcommand{\bea}{\begin{eqnarray}}
\newcommand{\eea}{\end{eqnarray}}

\begin{document}
\title[Balanced modular parameterizations]{Balanced modular
  parameterizations}


\author{Tim Huber, Danny Lara and Esteban Melendez}
\address{Department of Mathematics, University of Texas - Pan American, 1201 West University Avenue,  Edinburg, Texas 78539,  USA}



\begin{abstract}
For prime levels $5 \le p \le 19$, sets of $\Gamma_{0}(p)$-permuted theta quotients are constructed that
generate the graded rings of modular forms of positive
integer weight for $\Gamma_{1}(p)$. An explicit formulation of the
permutation representation and several applications are given,
including a new representation for the number of $t$-core partitions. 
The $\Gamma_{0}(p)$-action induces coefficient symmetries within
representations for modular forms and invariance
subgroups for coupled systems of
differential equations. 
The symmetry for levels $p = 5,7,11$ is linked to the Kleinian
automorphism groups.
\end{abstract}

\maketitle





\vspace{-0.2in}
\section{Introduction}

The graded ring of modular forms for a finite index subgroup of the full
modular group is isomorphic to a polynomial ring in two or
more generators modulo a finite set of relations \cite{MR0337993},
\cite[p. 249]{MR1221103}.
Certain classical polynomial representations for modular forms exhibit
coefficient symmetry. For example, the \textit{Klein polynomials}, whose roots
encode distinguished points of the stereographically projected circumsphere
for a regular icosahedron, are symmetric in absolute value about the
middle coefficients
 \begin{align} \label{kl1}
K_{v}(\Lambda) = \left (1-11 \Lambda -\Lambda^2\right)^5&, \qquad
K_{e}(\Lambda) = 1 + 228 \Lambda + 494 \Lambda^{2} - 228 \Lambda^{3} + \Lambda^{4}, \\
K_{f}(\Lambda) = 1- 522&\Lambda  - 10005 \Lambda^2- 10005\Lambda^4 +
522\Lambda^5+ \Lambda^6. \label{kl2}
  \end{align}
The polynomials $K_{e}(\Lambda)$ and $K_{f}(\Lambda)$, encoding the
edge and face points,
correspond to representations
for Eisenstein series
in terms of two
modular parameters of level five
\begin{align}
  \label{eq:24}
B^{20} K_{e}&(\Lambda) = 1+ 240 \sum_{n=1}^{\infty} \frac{n^{3} q^{n}}{1 -
    q^{n}}, \quad B^{30} K_{f}(\Lambda) = 1 - 504 \sum_{n=1}^{\infty} \frac{n^{5} q^{n}}{1 -
    q^{n}}, \quad |q|<1,\\ 
%
& A^{5}(q)
= q\frac{(q;q)_{\infty}^{2}}{(q^{2}, q^{3}; q^{5})_{\infty}^{5}},
\qquad B^{5}(q)= \frac{(q;q)_{\infty}^{2}}{(q, q^{4}; q^{5})_{\infty}^{5}},
\qquad \Lambda = A^5/B^5, \label{ab} \end{align}
where $(a;q)_{n} = \prod_{k=0}^{n-1} 1- aq^{k}$ and $(a_1, a_2,
\ldots, a_r ; q)_{n} = \prod_{j-1}^{r} (a_j; q)_{n}$ for $n \in \Bbb N
\cup {\infty}$.
These are special cases of more general symmetric parameterizations in $A^{5}, B^{5}$ from \cite{qeis}.
At level seven, certain modular forms are symmetric functions of the parameters
\begin{align}
  \label{eq:68}
  x = q \frac{(q^{2}, q^{5}, q^{7}, q^{7}; q^{7})_{\infty}}{(q^{3},
    q^{4}; q^{7})_{\infty}^{2}}, \quad y =  -q \frac{(q, q^{6}, q^{7}, q^{7}; q^{7})_{\infty}}{(q^{2},
    q^{5}; q^{7})_{\infty}^{2}},\quad z = \frac{(q^{3}, q^{4}, q^{7}, q^{7}; q^{7})_{\infty}}{(q,
    q^{6}; q^{7})_{\infty}^{2}}.
\end{align} 
For example, the Hecke Eisenstein series twisted, respectively by the Jacobi symbol
and trivial character $\chi_{1,7}$ modulo seven have the symmetric
representations  \cite{cooper_57,deq_sept}
\begin{align}
  \label{eq:70}
x+y+z = 1 + 2 &\sum_{n=1}^{\infty} \left ( \frac{n}{7} \right )
\frac{q^{n}}{1 - q^{n}}, \quad 
x^{2} + y^{2} + z^{2} = 1 +4\sum_{{n=1}}^{\infty} \frac{\chi_{1,7}(n)nq^{n}}{1-q^{n}},
\end{align}
and are connected to one other by the curious identity \cite[Ch. 21, Entry 5 (i)]{ramnote}
\begin{align}
  \label{eq:87}
  (x+y+z)^{2} = x^{2} + y^{2} + z^{2}.
\end{align}
This identity results from the Klein quartic equation,
with symmetric form \cite{deq_sept}
\begin{align}
  \label{eq:86}
  xy + xz + y z = 0.
\end{align}

Many formulations for modular forms of prime level, not necessarily in symmetric form, appear in the work of Klein and
Ramanujan.
Because the symmetry may appear incidental, no unified study of
symmetric modular parameterizations or extensions to other settings have been undertaken. In the
present work, such
symmetric constructions are shown to be special cases of more general balanced
parameterizations for modular forms of prime level $p$ with $5 \le p \le 19$. The coefficient symmetry
is a hallmark of a special collection of theta
quotients generating the graded ring of modular forms on
\begin{align}
  \label{eq:25}
\Gamma_{1}(p) &= \left \{
  \begin{pmatrix}
    a & b \\ c & d 
  \end{pmatrix} \in PSL(2, \Bbb Z) \mid c \equiv 0,\ a\equiv d \equiv
  1  \pmod{p} \right\}.
\end{align}
In each case, the polynomial generators are permuted up to a change
of sign by
\begin{align}
  \label{eq:74}
  &\Gamma_{0}(p) = \left \{
  \begin{pmatrix}
    a & b \\ c & d 
  \end{pmatrix} \in PSL(2, \Bbb Z) \mid c \equiv 0 \pmod{p} \right\}.
\end{align}
We show that balanced polynomial representations for modular forms on $\Gamma_{0}(p)$ result from a
nontrivial permutative action on generators for the graded rings of modular forms for $\Gamma_{1}(p)$
induced by modular transformation formulas. 
In addition to the list of permuted generators of level five and seven
already presented, a corresponding set of $\Gamma_{0}(11)$-permuted
 generators for the graded ring of forms on
$\Gamma_{1}(11)$ will be given by
\begin{align}
  \label{eq:84}
 &\frac{(q^{4}, q^{7}, q^{11}, q^{11}; q^{11})_{\infty}}{(q,
    q^{10}, q^{2}, q^{9}; q^{11})_{\infty}},\quad  q \frac{(q^{5}, q^{6}, q^{11}, q^{11}; q^{11})_{\infty}}{(q^{3},
    q^{8}, q^{4}, q^{7}; q^{11})_{\infty}},\quad q^{2} \frac{(q, q^{10}, q^{11}, q^{11}; q^{11})_{\infty}}{(q^{3},
    q^{8}, q^{5}, q^{6}; q^{11})_{\infty}}, \\ &\qquad \quad \qquad q \frac{(q^{3}, q^{8}, q^{11}, q^{11}; q^{11})_{\infty}}{(q^{2},
    q^{9}, q^{4}, q^{7}; q^{11})_{\infty}}, \qquad q \frac{(q^{2}, q^{9}, q^{11}, q^{11}; q^{11})_{\infty}}{(q^{5},
    q^{6}, q, q^{10}; q^{11})_{\infty}}.
\end{align}
A similarly permuted set of generators for the graded ring of forms for
$\Gamma_{1}(13)$ is
\begin{align}
  \label{eq:84}
 &\frac{(q^{6}, q^{7}, q^{13}, q^{13}; q^{13})_{\infty}}{(q,
    q^{12}, q^{3}, q^{10}; q^{13})_{\infty}},\quad  q \frac{(q^{5}, q^{8}, q^{13}, q^{13}; q^{13})_{\infty}}{(q^{3},
    q^{10}, q^{4}, q^{9}; q^{13})_{\infty}},\quad \ \ \ q \frac{(q, q^{2}, q^{11}, q^{13}; q^{13})_{\infty}}{(q,
    q^{12}, q^{4}, q^{9}; q^{13})_{\infty}}, \\  q &\frac{(q^{4}, q^{9}, q^{13}, q^{13}; q^{13})_{\infty}}{(q^{2},
    q^{11}, q^{5}, q^{8}; q^{13})_{\infty}}, \quad q^{2} \frac{(q^{3}, q^{10}, q^{13}, q^{13}; q^{13})_{\infty}}{(q^{5},
    q^{8}, q^{6}, q^{7}; q^{13})_{\infty}}, \quad q^{2} \frac{(q, q^{12}, q^{13}, q^{13}; q^{13})_{\infty}}{(q^{2},
    q^{11}, q^{6}, q^{7}; q^{13})_{\infty}}.
\end{align}
A set of $\Gamma_{0}(17)$-permuted generators for the graded ring
of forms on $\Gamma_{1}(17)$ is
\begin{align}
  \label{eq:85}
   &\frac{(q^{8}, q^{9}, q^{17}, q^{17}; q^{17})_{\infty}}{(q^{2},
    q^{15}, q^{3}, q^{14}; q^{17})_{\infty}},\quad  q \frac{(q^{5}, q^{12}, q^{17}, q^{17}; q^{17})_{\infty}}{(q^{3},
    q^{14}, q^{4}, q^{13}; q^{17})_{\infty}},\quad \ \ \ q^{3} \frac{(q, q^{16}, q^{17}, q^{17}; q^{17})_{\infty}}{(q^{4},
    q^{13}, q^{6}, q^{11}; q^{17})_{\infty}}, \\  q^{2} &\frac{(q^{7}, q^{10}, q^{17}, q^{17}; q^{17})_{\infty}}{(q^{6},
    q^{11}, q^{8}, q^{9}; q^{17})_{\infty}}, \quad q^{3} \frac{(q^{2}, q^{15}, q^{17}, q^{17}; q^{17})_{\infty}}{(q^{5},
    q^{12}, q^{8}, q^{9}; q^{17})_{\infty}}, \quad q\frac{(q^{3}, q^{14}, q^{17}, q^{17}; q^{17})_{\infty}}{(q,
    q^{16}, q^{5}, q^{12}; q^{17})_{\infty}}, \\ & \quad \  \qquad \qquad
  q \frac{(q^{4}, q^{13}, q^{17}, q^{17}; q^{17})_{\infty}}{(q,
    q^{16}, q^{7}, q^{10}; q^{17})_{\infty}}, \quad q\frac{(q^{6}, q^{11}, q^{17}, q^{17}; q^{17})_{\infty}}{(q^{2},
    q^{15}, q^{7}, q^{10}; q^{17})_{\infty}}.
\end{align}
Finally, a set of permuted generators for
the graded ring of forms for $\Gamma_{1}(19)$ is
{\allowdisplaybreaks \begin{align}
  \label{eq:85}
   &\frac{(q^{8}, q^{11}, q^{9}, q^{10}, q^{19}, q^{19}; q^{19})_{\infty}}{(q^{3},
    q^{16}, q^{4}, q^{15},q^{5}, q^{14}; q^{19})_{\infty}},\qquad  q
  \frac{(q^{2}, q^{17}, q^{7}, q^{12}, q^{7}, q^{12}; q^{19})_{\infty}}{(q,
    q^{18}, q^{4}, q^{15},q^{6}, q^{13}; q^{19})_{\infty}}, \\ & q \frac{(q^{3}, q^{16}, q^{9}, q^{10}, q^{19}, q^{19}; q^{19})_{\infty}}{(q, q^{18}, q^{5}, q^{14}, q^{8}, q^{11}; q^{19})_{\infty}}, \qquad q\frac{(q^{4}, q^{15}, q^{7}, q^{12}, q^{19}, q^{19}; q^{19})_{\infty}}{(q^{2},
    q^{17}, q^{5}, q^{14}, q^{6}, q^{13}; q^{19})_{\infty}}, \\   q^{5} &\frac{(q, q^{18}, q^{3}, q^{16}, q^{19}, q^{19}; q^{19})_{\infty}}{(q^{6},
    q^{13}, q^{8}, q^{11}, q^{9}, q^{10}; q^{19})_{\infty}}, \qquad q^{2} \frac{(q^{4}, q^{15}, q^{5}, q^{14}, q^{19}, q^{19}; q^{19})_{\infty}}{(q^{2},
    q^{17}, q^{7}, q^{12}, q^{8}, q^{11}; q^{19})_{\infty}}, \\ 
  q^{2}& \frac{(q, q^{18}, q^{6},q^{13}, q^{19}, q^{19}; q^{19})_{\infty}}{(q^{2},
    q^{17}, q^{3}, q^{16},q^{9}, q^{10}; q^{19})_{\infty}}, \qquad q^{2}\frac{(q^{5}, q^{14}, q^{8}, q^{11},q^{19}, q^{19}; q^{19})_{\infty}}{(q^{4},
    q^{15}, q^{7}, q^{12},q^{9}, q^{10}; q^{19})_{\infty}}, \\  & \qquad \qquad \qquad \qquad q\frac{(q^{2}, q^{17}, q^{6}, q^{13},q^{19}, q^{19}; q^{19})_{\infty}}{(q,
    q^{18}, q^{3}, q^{16},q^{7}, q^{12}; q^{19})_{\infty}}. \label{iwb}
\end{align}}

The distinguishing feature of each claimed set of polynomial generators
appearing here is that any modular form of positive integer weight on a subgroup containing
$\Gamma_{0}(p)$ enjoys a polynomial representation exhibiting a certain coefficient symmetry.
\begin{thm} \label{struct}
  For each prime $p$ with $5\le p\le 19$, let $\{x_{k}\}_{k=1}^{\frac{p-1}{2}}$ denote the given  generators for the graded ring of modular forms for $\Gamma_{1}(p)$.
Any modular form of weight $k$ for a subgroup of $PSL(2,\Bbb Z)$ containing $\Gamma_{0}(p)$
is representable in the form
\begin{align}
  \label{eq:51}
\sum_{k_{1}+\cdots+k_{(p-1)/2}=k}a_{k_{1},\ldots,k_{(N-1)/2}}\sum_{\sigma\in S_{(p-1)/2}}\epsilon_{\sigma(k_{1}),\ldots,\sigma(k_{(p-2)/2})}x_{1}^{\sigma(k_{1})}\cdots x_{(p-1)/2}^{\sigma(k_{(p-1)/2})},  
\end{align}
over integers $k_{i} \ge 0$, $a_{k_{1},\ldots,k_{(N-1)/2}}\in \Bbb C$, $S_{n}$ the symmetric group on $n$ elements, and  
$$\epsilon_{\sigma(k_{1}),\ldots,\sigma(k_{(N-2)/2})}\in\{\pm1\}.$$
In particular, the coefficient of each monomial in \eqref{eq:51} agrees in absolute
value with the coefficient of any other monomial obtained through a permutation
of the exponents.
\end{thm}
The coefficient symmetry is induced by transformation formulas
satisfied by the generators. To describe the
symmetry exhibited by forms of level $p$ in terms of the generators, let $\gamma \in PSL(2, \Bbb Z)$ act on the upper half plane by M\"obius transformation, and define the diamond operator on a modular form $f$ of weight $k$ on $\Gamma_{1}(p)$ by
\begin{align} 
\langle \gamma \rangle(f) =(\gamma_{21} \tau + \gamma_{22})^{-k}f(\gamma \tau), \qquad  \label{eq:69}
 \gamma =  (\gamma_{11}, \gamma_{12}; \gamma_{21}, \gamma_{22}) 
\in \Gamma_{0}(p).
\end{align}
\begin{thm} \label{pmr}
The generators for $\Gamma_{1}(p)$ from \eqref{ab}--\eqref{iwb} are
permuted up to change of sign by $\Gamma_{0}(p)$ under action by $\langle \cdot \rangle$, with permutation representation $( \Bbb Z / p \Bbb Z)^{*}/\{\pm 1\}$.
\end{thm}


A common thread linking the symmetric constructions is the Hecke
Eisenstein series for $\Gamma_{0}(p)$ twisted by Dirichlet character $\chi$
modulo $p$, defined for weight $k \in \Bbb N$ by
\begin{align} \label{eisdef}
  E_{k,\chi}(\tau) = 1 +  \frac{2}{L(1 - k, \chi)} \sum_{n=1}^{\infty} \chi(n) \frac{n^{k-1} q^{n}}{1 - q^{n}}, \qquad q = e^{2 \pi i \tau},
\end{align}
where $L(1 - k, \chi)$ is the analytic continuation of the associated
Dirichlet $L$-series and $\chi(-1) = (-1)^{k}$. We prove the
above claims in the form of Theorems
\ref{sf}--\ref{lem:mod} by systematically expressing generating theta quotients
in terms of Eisenstein series.

In Section \ref{s_eis}, notable representations are highlighted for modular
forms in terms of the generating parameters. In particular, the
product of elements in each generating set is closely related to the
generating functions for $p$-cores. A new convolution representation for $p$-cores
is given in Corollary \ref{cortc} for $5 \le p \le 19$.
We also derive coupled
systems of differential equations for the generators of level
$5 \le p \le 19$. Each system is invariant under
action by $\Gamma_{0}(p)$. The differential systems of level five and seven
appeared recently in \cite{MR3111558,deq_sept}. Their coefficient invariance is a common feature of the coupled
differential systems derived for each set of generators of level $5
\le p \le 19$.

\begin{thm} \label{th:11}  Let $E_{2}(q) = 1 - 24 \sum_{n=1}^{\infty}
  \frac{n q^{n}}{1 - q^{n}}$ and $\mathscr{P} = E_{2}(q^{5})$. Then 
  \begin{align}
\label{deqa}    q \frac{d}{dq} A &= \frac{1}{60} A \Bigl (-5 A^{10}-66 A^5 B^5+7 B^{10}+5 \mathscr{P} \Bigr ), \\ 
\label{deqb} q \frac{d}{dq} B &= \frac{1}{60} B \Bigl (-5 B^{10}+66 A^5 B^5+7 A^{10}+5 \mathscr{P} \Bigr ), \\ 
q \frac{d}{dq} \mathscr{P} = \frac{5}{12}
\Bigl ( \mathscr{P}^{2}&-B^{20}+ 12 B^{15}A^5 - 14 B^{10}A^{10}- 12 B^5
A^{15}- A^{20}\Bigr ), \label{dpp}
  \end{align}
\end{thm}
\begin{thm} \label{th:12} Let $\mathcal{P}(q) = E_{2}(q^{7})$. Then 
    \begin{align} \label{deqx}
    q\frac{d}{dq}x &= \frac{x}{12} \Bigl ( 5 y^{2} + 5z^{2} - 7x^{2}+
    20 yz +52 xy  + 7 \mathcal{P} \Bigr ), \\ q\frac{d}{dq}y
    &=\frac{y}{12} \Bigl ( 5z^{2 } + 5x^{2}  - 7 y^{2} + 20 xz +52yz +
    7 \mathcal{P} \Bigr ), \label{deqy} \\ q\frac{d}{dq}z &=
    \frac{z}{12} \Bigl ( 5x^{2} + 5y^{2} - 7 z^{2} + 20 xy+ 52xz + 7
    \mathcal{P}  \Bigr ), \label{deqz}\\ 
     q \frac{d}{dq} \mathcal{P}(q)  = \frac{7}{12} \Bigl (\mathcal{P}^2 -&x^4-4 x^3 y-12 x y^3-y^4-12 x^3 z-4 y^3 z-4 x z^3-12 y z^3-z^4 \Bigr ). \nonumber
  \end{align} 
\end{thm}
Theorem \ref{bu} encodes in concise form the differential equations
from Theorems \ref{th:11} and \ref{th:12} and provides coupled systems
for the generators corresponding to higher prime levels. These systems are analogous to those for modular parameters
of lower levels \cite{hhahn2,huber_proc,cubic,maier-2008} and to Ramanujan's differential system for
Eisenstein series \cite{Raman1}
{\allowdisplaybreaks \begin{equation} 
\label{rdiff1}  q \frac{dE_{2}}{dq} = \frac{E_{2}^{2} - E_{4}}{12}, \qquad  q \frac{dE_{4}}{dq} = \frac{E_{2}E_{4} - E_{6}}{3}, \qquad q \frac{dE_{6}}{dq} = \frac{E_{2}E_{6} - E_{4}^{2}}{2},
\end{equation}}
where the normalized Eisenstein series $E_{k} = E_{k}(q)$ for $PSL(2, \Bbb Z)$  are defined by
\begin{align} \label{eis_full}
  E_{2k}(q) = 1 + \frac{2}{\zeta(1 - 2 k)} \sum_{n=1}^{\infty} \frac{n^{2k-1} q^{n}}{1 - q^{n}},
\end{align}
and where $\zeta$ is the analytic continuation of the Riemann
$\zeta$-function. In fact, the last equations of Theorem \ref{th:11} and
\ref{th:12} are equivalent to the first equation of
\eqref{rdiff1}. In view of Theorem \ref{gtd}, the
coupled systems corresponding to each level may be viewed as systems
of equations for certain linear combinations of twisted Eisenstein
series. 


The final part of the paper, Section \ref{s_higher}, connects the permuted generators to results in classical representation
theory. In particular, we derive parameterizations for the Klein
quartic and analogous identities for the
generating parameters at higher levels.

\section{Elliptic modular preliminaries}

Before embarking on proofs of the claims in the last section, we
introduce some fundamental notions from the theory of elliptic modular
forms. Let $\mathcal{M}_{k}(\Gamma)$ denote the vector space of weight $k$
modular forms for $\Gamma\subseteq PSL(2,\Bbb Z)$. For a given prime $p$, the $(p-1)/2$ linearly independent Eisenstein series of weight one and primitive
character $\chi$ are known to generate a subspace, called the
Eisenstein subspace, of modular
forms of weight one for $\Gamma_{1}(p)$ \cite[Theorem
4.8.1]{dish}. 
For each prime $p$ with $5\le p \le 19$, \cite{buzzard,MR2289048}
\begin{align}
  \dim \left ( \mathcal{M}_{1}(\Gamma_{1}(p)) \right ) = \frac{p-1}{2}.
\end{align}
Hence, the set of weight one Eisenstein series of odd primitive character
modulo $p$ form a basis for $\mathcal{M}_{1}(\Gamma_{1}(p))$ over
$\Bbb C$. The claimed symmetric generators of this paper originate
from theta function expansions for certain linear combinations of
Eisenstein series of weight one. To formulate the change of bases from Eisenstein series to the
 permuted bases of products from
\eqref{ab}--\eqref{iwb}, we first derive, in Theorem \ref{sf}, representations for sums
of Eisenstein series in terms of the Dedekind eta function, $\eta(\tau)
= q^{1/24}(q;q)_{\infty}  $, a weight
$1/2$ modular form for $SL(2, \Bbb Z)$ with
multiplier given explicitly by \cite[p. 51]{knopp}.
The relevant Eisenstein series representations will also involve the Jacobi theta function  
\begin{align} \label{jt}
    \theta_{1}(z \mid q) = -iq^{1/8} \sum_{n=-\infty}^{\infty} (-1)^{n} q^{n(n+1)/2} e^{(2n+1)iz}, 
  \end{align}
an odd function of $z$ with a simple zero at the origin such that
\cite[p. 489]{witwat}
\begin{align} \label{another1}
 \frac{\theta_{1}'}{\theta_{1}} ( z\mid q)  &= \cot z + 4
 \sum_{n=1}^{\infty} \frac{q^{n}}{1 - q^{n}} \sin 2 n z \\
&= i - 2 i \sum_{n=1}^{\infty} \frac{q^{n} e^{2 i z}}{ 1 - q^{n} e^{2 i z}} + 2
 i \sum_{n=0}^{\infty} \frac{q^{n} e^{- 2 i z}}{1 - q^{n} e^{- 2 i
     z}}. \label{another}
\end{align}
 Our subsequent calculations require the following easily verified functional equations
\begin{align}
  \label{eq:5}
  \theta_{1}(z + n \pi) = (-1)^{n} \theta_{1}(z \mid q), \qquad
  \theta_{1}(z + n \pi  \tau \mid q) = (-1)^{n} q^{-n^{2}/2} e^{-2inz}
  \theta_{1}(z \mid q).
\end{align}
Product representations result from the Jacobi Triple Product expansion given by \cite{witwat}
\begin{align} \label{jtp}
  \theta_{1}(z \mid q) = - i q^{1/8} e^{iz} (q;q)_{\infty} (q e^{2 i z}; q)_{\infty} (e^{- 2 i z}; q)_{\infty}.
\end{align}
In particular, we will make use of the special case 
\begin{align} \label{pw}
  \theta_{1}(r \pi \tau \mid q^{p}) = i q^{-r} q^{p/8} q^{r/2} (q^{r}, q^{p-r}, q^{p}; q^{p})_{\infty}.
\end{align}
By differentiating \eqref{jtp} at the origin, we obtain 
\begin{align} \label{yi1}
  \theta_{1}'(q) := \lim_{z \to 0} \frac{\theta_{1}(z \mid q)}{z} = 2 q^{1/8}(q;q)_{\infty}^{3}.  
\end{align}

To extend the bases of weight one forms on $\Gamma_{1}(p)$ to homogeneous
representations for any positive integer weight, it
suffices to generate modular forms up to weight three. For prime levels $N
\ge 5$, this is proven in \cite{MR1992801}, and for general
$N$ in \cite{rustom}. The situation is better for principal
congruence subgroups of level $N$, where weight $1$ suffices \cite{MR2904927}. 
\begin{lem} \label{lem:nm} Denote by $ \mathcal{M}_{k}(\Gamma)$ the
  $\Bbb C$-vector space of weight $k$ modular forms for the congruence
  subgroup
  $\Gamma$, and let $\mathcal{M}(\Gamma) =
  \bigoplus_{k=1}^{\infty} \mathcal{M}_{k}(\Gamma)$ be the
  corresponding graded ring.
  \begin{enumerate}
  \item   For $N \ge 5$, any algebra containing
    $\mathcal{M}_{k}(\Gamma_{1}(N))$, for $k \le 3$, contains
    $\mathcal{M}(\Gamma_{1}(N))$.
  \item   For $N \ge 3$, any algebra containing  $\mathcal{M}_{1}(\Gamma(N))$ contains $\mathcal{M}(\Gamma(N))$.
  \end{enumerate}
\end{lem}

To show that $\Gamma_{0}(p)$ acts as indicated on the
generating quotients, it will be convenient to apply transformation
formulas for special values of the Jacobi theta function \eqref{jt}
in the form of those for theta constants of odd order $k$ and index $\ell$, defined by \cite{MR1850752}  
\begin{align} \label{thetac}
\varphi_{k,\ell}(\tau) = \theta [\chi_{\ell,k}](0,
k\tau),  \qquad \chi_{\ell,k} = \left [ \begin{array}{c} \frac{2 \ell -1}{k} \\
    1 \end{array} \right],\qquad 1 \le \ell \le \frac{k-1}{2},
\end{align}
in turn, constructed from theta constants of characteristic $[\epsilon, \epsilon']\in \Bbb R^{2}$ 
\begin{align} \label{fark}
   \theta \left [ \begin{array}{c} \epsilon \\ \epsilon' \end{array} \right ](z, \tau) = \sum_{n \in \Bbb Z} \exp 2 \pi i \left \{ \frac{1}{2} \left ( n + \frac{\epsilon}{2} \right )^{2}\tau + \left ( n + \frac{\epsilon}{2} \right ) \left ( z + \frac{\epsilon'}{2} \right ) \right \}.
\end{align}
\begin{thm}{\cite[pp. 215-219]{MR1850752}}\label{far}
For odd positive integers $k \ge 3$, let 
\begin{align}
        \label{eq:26}
        \mathcal{V}_{k}(\tau) = \left[ \theta \left[ 
    \begin{array}{c}
      (k-2)/k \\ 1
    \end{array} \right ](k \tau),  \theta \left[ 
    \begin{array}{c}
      (k-4)/k \\ 1
    \end{array} \right ](k \tau), \ldots,  \theta \left[ 
    \begin{array}{c}
      1/k \\ 1
    \end{array} \right ](k \tau) \right ]^{T}. 
\end{align}
Then $\mathcal{V}_{k}$ is a vector-valued form of weight $1/2$ on
$PSL(2, \Bbb Z)$ inducing a representation 
$$\pi_{k}: PSL(2, \Bbb Z)\to PGL((k-1)/2, \Bbb C) \quad \hbox{via}\quad 
\mathcal{V}_{k}(\gamma \tau) = (\gamma_{21}\tau + \gamma_{22})^{1/2}\pi_{k}(\gamma) \mathcal{V}_{k}(\tau) $$
determined by the images of generators for $SL(2, \Bbb Z)$, $S = (0, -1; 1 ,0)$, $T = (1,1;0,1)$,
\begin{align*}
  \mathcal{V}_{N}(T\tau) = \mathcal{V}_{N}(\tau +1) = \pi_{N}(T)
  \mathcal{V}_{N}(\tau), \quad   \mathcal{V}_{N}(S\tau) = \mathcal{V}_{N}(-1/\tau) = \tau^{1/2} \pi_{N}(S) \mathcal{V}_{N}(\tau),
\end{align*}
%
where the matrices $\pi_{N}(S)$ and $\pi_{N}(T)$  have
$(\ell,j)$th entry, for $1 \le \ell, j \le (N-1)/2$,
\begin{align}
  \label{eq:31}
  \{ \pi_{N}(T) \}_{(\ell,j)} &=
  \begin{cases}
    \exp \left ( \frac{(N - 2 \ell)^{2} \pi i}{4N} ) \right ), & \ell = j, \\ 
0,& else,
  \end{cases} \\
 \{ \pi_{N}(S) \}_{(\ell,j)} &= \frac{\left(1+e^{\frac{ (2 j-N)
         (N-2 \ell)}{k}\pi i}\right) \exp \left(\frac{\left(j (-2
         N+4 \ell+2)+N^2-2 (N+1) \ell \right)}{2 N} \pi
     i\right)}{\sqrt{i N}}. \label{eq:31a}
\end{align}
\end{thm}

Equivalent transformation formulas figure prominently in Klein's representation
of the automorphism group of the icosahedron
\cite{MR2133308,MR0080930,MR1315530}; in the septic extension to the
Klein quartic \cite{MR1722419}; as well as in Klein's level $11$ analysis of the
symmetries of the Klein cubic \cite{klein_el}. Our formulation of
the quartic and higher level extensions are equivalent to quadratic
relations derivable from
a classical theta function identity \cite[p. 518]{witwat}
  \begin{align} \label{eq:13}
    \left ( \frac{\theta_{1}'}{\theta_{1}} \right )' (x \mid q) -
    \left ( \frac{\theta_{1}'}{\theta_{1}} \right )' (y \mid q) =
    \frac{\theta_{1}'(0 \mid q)^{2} \theta_{1}(x - y \mid q) \theta_{1}(x+y \mid q)}{\theta_{1}^{2}(x \mid q) \theta_{1}^{2}(y \mid q)}.
  \end{align}

\section{Eisenstein expansions for permuted bases} \label{huj}

The goal of this section is to formulate and prove claims made in the Introduction for generators of the graded rings of modular forms on
$\Gamma_{1}(p)$. We construct the generators and prove the claimed permutative action of $\Gamma_{0}(p)$. 
In Theorem \ref{sf},  product expansions are derived for the normalized
sums of weight one Eisenstein series twisted by the odd primitive
Dirichlet characters modulo $p$. 
By considering the by $\Gamma_{0}(p)$-orbit of these series under modular transformation, we derive
in Theorems \ref{gtd} and \ref{te}, bases for
the weight one forms on $\Gamma_{1}(p)$ and explicitly characterize
the permutative action by
$\Gamma_{0}(p)$. Theorem \ref{gtd1} demonstrates that the Eisenstein bases
from Theorem \ref{gtd} are expressible as quotients of theta
functions. Finally, Theorem \ref{lem:mod} proves the corresponding theta
quotients generate the graded ring of modular forms for
$\Gamma_{1}(p)$ of positive integer weight.
\begin{thm} \label{sf}
  Define $E_{\chi,k}(\tau)$ as in \eqref{eisdef}. For each prime $5\le p\le 19$, let
  \begin{align} \label{defe}
    \mathcal{E}_{p}(\tau) = \frac{2}{p-1}\sum_{\chi(-1)=-1} E_{\chi,1}(\tau),
  \end{align}
where the sum is over the odd primitive Dirichlet characters modulo $p$. Then 
\begin{align} \label{lefm}
  \mathcal{E}_{5}(\tau) &=   \frac{(q;q)_{\infty}^{2}}{(q, q^{4};
    q^{5})_{\infty}^{5}}, 
\qquad \qquad \qquad \ 
\mathcal{E}_{7}(\tau) =  \frac{(q^{3}, q^{4}, q^{7}, q^{7}; q^{7})_{\infty}}{(q,
    q^{6}; q^{7})_{\infty}^{2}}, \\ 
\mathcal{E}_{11}(\tau) &=  \frac{(q^{4}, q^{7}, q^{11}, q^{11}; q^{11})_{\infty}}{(q,
    q^{10}, q^{2}, q^{9}; q^{11})_{\infty}}, \qquad
\mathcal{E}_{13}(\tau) = \frac{(q^{6}, q^{7}, q^{13}, q^{13}; q^{13})_{\infty}}{(q,
    q^{12}, q^{3}, q^{10}; q^{13})_{\infty}}, \\ 
\mathcal{E}_{17}(\tau) &=  \frac{(q^{8}, q^{9}, q^{17}, q^{17}; q^{17})_{\infty}}{(q^{2},
    q^{15}, q^{3}, q^{14}; q^{17})_{\infty}}, \qquad 
\mathcal{E}_{19}(\tau) = \frac{(q^{8}, q^{11}, q^{9}, q^{10}, q^{19}, q^{19}; q^{19})_{\infty}}{(q^{3},
    q^{16}, q^{4}, q^{15},q^{5}, q^{14}; q^{19})_{\infty}}.
\end{align}
\end{thm}
\begin{proof}
To prove each identity, we use the fact that the sum of the residues
of an elliptic function on its period parallelogram is zero
\cite{witwat}.  The challenge lies in writing down the relevant elliptic
functions. We begin by proving the leftmost equation of \eqref{lefm}. Let 
\begin{align}
  f_{5}(z)=\frac{e^{-2iz}\theta_{1}^{3}(z-\pi\tau\mid q^{5})}{\theta_{1}^{2}(z\mid q^{5})\theta_{1}(z+2\pi\tau\mid q^{5})}.
\end{align}
Apply \eqref{eq:5} to verify that $f_{5}(z)$ is an elliptic function with periods $\pi$ and $5 \pi \tau$. From corresponding properties of the Jacobi theta function, observe that $f_{5}(z)$ has a simple pole at $z = -2\pi\tau$ and a double pole at $z = 0$. The residue of $f_{5}(z)$ at $z = - 2 \pi\tau$ is 
\begin{align}
	\lim_{z\to-2\pi\tau}\frac{(z+2\pi\tau)}{\theta_{1}(z+2\pi\tau\mid q^{5})}\lim_{z\to-2\pi\tau}\frac{e^{-2iz}\theta_{1}^{3}(z-\pi\tau\mid q^{5})}{\theta_{1}^{2}(z\mid q^{5})} 
	=	\frac{q^{2}}{\theta_{1}'(q^{5})}\cdot\frac{\theta_{1}^{3}(-3\pi\tau\mid q^{5})}{\theta_{1}^{2}(-2\pi\tau\mid q^{5})}.
 \end{align}
The residue of $f_{5}(z)$ at $z=0$ is 
\begin{align}
\lim_{z\to0}(z^{2}f_{5}(z))'	&=	\lim_{z\to0}\Bigl(z^{2}f(z)\Bigr)\left(\frac{2}{z}+\frac{f_{5}'(z)}{f_{5}(z)}\right) \\
	&=	\left(\lim_{z\to0}\frac{z^{2}}{\theta_{1}^{2}(z\mid q^{5})}\right)\left(\lim_{z\to0}\frac{e^{-2iz}\theta_{1}^{3}(z-\pi\tau\mid q^{5})}{(z+2\pi\tau\mid q^{5})}\right)\left(\lim_{z\to0}\frac{2}{z}+\frac{f_{5}'(z)}{f_{5}(z)}\right) \\
	&=
        \frac{-1}{\theta_{1}'(q^{5})^{2}}\left(\frac{\theta_{1}^{3}(\pi\tau\mid
            q^{5})}{\theta_{1}(2\pi\tau\mid
            q^{5})}\right)\left(\lim_{z\to0}\frac{2}{z}+\frac{f_{5}'(z)}{f_{5}(z)}\right). 
 \end{align}
Since the sum of the residues of $f_{5}(z)$ is zero, we obtain from \eqref{pw}
\begin{align} \label{helb}
-2i\frac{(q^{2}, q^{3},q^{5}; q^{5})_{\infty}^{2}}{(q, q^{4}, q^{5}; q^{5})_{\infty}^{3}}(q^{5};q^{5})_{\infty}^{3}
=\lim_{z\to0}\frac{2}{z}+\frac{f_{5}'(z)}{f_{5}(z)}.
 \end{align}
By applying identities \eqref{another1}-\eqref{another}, and the Laurent expansion for
$\cot z$, we derive
\begin{align}
  \label{eq:95}
  \lim_{z\to 0}\frac{2}{z}+\frac{f_{5}'(z)}{f_{5}(z)} &= \lim_{z\to 0}
  \left (
  \frac{2}{z} - 2 \frac{\theta_{1}'}{\theta_{1}}(z \mid q^{5})\right ) - 2 i -
 3 \frac{\theta_{1}'}{\theta_{1}}(\pi \tau \mid q^{5}) +
 \frac{\theta_{1}'}{\theta_{1}}(2\pi \tau \mid q^{5})  \\ &= - 2 i -
 3 \frac{\theta_{1}'}{\theta_{1}}(\pi \tau \mid q^{5})
 -\frac{\theta_{1}'}{\theta_{1}}(2\pi \tau \mid q^{5}) = -2i - 
 \sum_{n=1}^{\infty} \frac{c_{n}q^{n}}{1 - q^{n}}, \label{ghj}
\end{align}
where, from \eqref{another}, $\{c_{n}\}_{n=1}^{\infty}$ is a periodic sequence modulo five such that 
\begin{align}
  \label{eq:96}
  c_{1} = -6i ,\qquad c_{2} =-2i, \qquad c_{3} =2i, \qquad c_{4} =6i,\qquad
  c_{5} = 0.
\end{align}
If we denote the two odd primitive Dirichlet
characters modulo five by
\begin{align}
  \langle \chi_{2,5}(n) \rangle_{n=0}^{4} = \langle 0, 1, i, -i, -1
  \rangle, \qquad \langle \chi_{4,5}(n) \rangle_{n=0}^{4} = \langle 0, 1, -i, i, -1 \rangle, \label{c2}
\end{align}
then, since, for $\chi$  non-principal modulo $p$, we may write \cite[pp. 136-137]{dish},
\begin{align}
    \label{eq:91}
    L(\chi, 0) = \sum_{n=0}^{p-1} \chi(n) \left ( \frac{1}{2} -
      \frac{n}{p} \right ),  
  \end{align}
it follows from  \eqref{c2} and \eqref{eq:91} that 
\begin{align}
  \label{eq:98}
  c_{n} = \frac{2i\chi_{2,5}(n)}{L(\chi_{2,5}, 0)} +  \frac{2i\chi_{4,5}(n)}{L(\chi_{4,5}, 0)} 
\end{align}
Therefore, from \eqref{defe} and identities \eqref{ghj}, \eqref{eq:98}, and \eqref{eisdef},
\begin{align}
  \label{eq:80}
  \lim_{z\to 0}\frac{2}{z}+\frac{f_{5}'(z)}{f_{5}(z)} &= 
-2i\mathcal{E}_{5}(q).
\end{align}
This completes the proof of the leftmost equation of \eqref{lefm}.  A formulation of rightmost equation of \eqref{lefm} from
\eqref{ht} is given in \cite[Eq. (3.23)]{MR1767652}.
For prime levels $7 \le p \le 19$, the claimed product expansions for
the Eisenstein sums $\mathcal{E}_{p}(\tau)$ may obtained by applying
the residue theorem with the elliptic functions $f_{p}(z)$ of period
$\pi, p\pi \tau$, defined by
\begin{align}
  f_{7}(z) &=e^{2iz}\frac{\theta_{1}^{2}(z+\pi \tau \mid q^{7}) \theta_{1}(z+2\pi \tau \mid q^{7})}{\theta_{1}^{2}(z\mid q^{7})\theta_{1}(z-3\pi \tau \mid q^{7})} \label{ht} ,\\
  f_{11}(z) &= e^{-2 i z} \frac{\theta_{1}(z - 2 \pi \tau \mid q^{11} )
    \theta_{1}(z - 3 \pi \tau \mid q^{11}) \theta_{1}(z - 5 \pi \tau
    \mid q^{11})}{\theta_{1}^{2}(z \mid q^{11}) \theta_{1}( z + \pi
    \tau \mid q^{11})}, \\ 
  f_{13}(z) &= e^{-2 i z} \frac{\theta_{1}(z - 3 \pi \tau \mid q^{13} )
    \theta_{1}(z - 4 \pi \tau \mid q^{13}) \theta_{1}(z - 5 \pi \tau
    \mid q^{13})}{\theta_{1}^{2}(z \mid q^{13}) \theta_{1}( z + \pi
    \tau \mid q^{13})},\\
  f_{17}(z) &= e^{-2 i z} \frac{\theta_{1}(z - 3 \pi \tau \mid q^{17}) \theta_{1}(z - 5 \pi \tau
    \mid q^{17}) \theta_{1}(z - 7 \pi \tau
    \mid q^{17})}{\theta_{1}^{2}(z \mid q^{17}) \theta_{1}( z + 2\pi
    \tau \mid q^{17})},\\
  f_{19}(z) &=e^{-2 i z} \frac{\theta_{1}(z - 4 \pi \tau \mid q^{19}) \theta_{1}(z - 5 \pi \tau
    \mid q^{19}) \theta_{1}(z - 7 \pi \tau
    \mid q^{19})}{\theta_{1}^{2}(z \mid q^{19}) \theta_{1}( z + 3\pi
    \tau \mid q^{19})},
\end{align}
each constructed by writing $\mathcal{E}_{p}(\tau)$ in terms of 
$(\theta_{1}'/\theta_{1})(k\pi\tau \mid q^{p})$, $1 \le k \le (p-1)/2$. 
\end{proof}
 We now construct alternative bases to the Eisenstein bases for $\mathcal{M}_{1}(\Gamma_{1}(p))$
by letting $\Gamma_{0}(p)$ act on the series $\mathcal{E}_{p}(\tau)$ and requiring that the first
nonzero coefficient in the $q$-expansion of the image of
$\mathcal{E}_{p}(\tau)$ under $\langle \cdot \rangle$ be $1$. Since
this action of $\langle \cdot \rangle$ 
depends only on the lower right
entry of $\gamma \in \Gamma_{0}(p)$, we list only this element in subsequent results.
\begin{thm} \label{gtd} 
Define $\langle \cdot \rangle$ by \eqref{eq:69} and
$\mathcal{E}_{p}$ by \eqref{defe}. For prime $5\le
p \le 19$, and a set of distinct elements
$\{a_{k,p}\}_{k=1}^{(p-1)/2} \subset (\Bbb Z / p \Bbb Z)^{*} / \{\pm
1\}$, there exists a basis decomposition
\begin{align}\label{rl}
  \mathcal{M}_{1}\left ( \Gamma_{1}(p) \right ) &=
  \bigoplus_{k=1}^{(p-1)/2} \mathbb{C} \langle a_{k,p} \rangle ( \mathcal{E}_{p} ).
  \end{align}
Moreover, if the constants $a_{k,p}$ are as follows, the basis elements $\langle a_{k,p} \rangle (
\mathcal{E}_{p} )$ of \eqref{rl}  are normalized so that the first nonzero
coefficient in their $q$-expansion is $1$:
$$(a_{k,5})_{k=1}^{2} = (1,2), \qquad (a_{k,7})_{k=1}^{3} = (1,2,3),
\qquad (a_{k,11})_{k=1}^{5} =
(1,2,3,5,7),$$ $$(a_{k,13})_{k=1}^{6} = (1,2,3,4,5,7),
\qquad (a_{k,17})_{k=1}^{8} = (1,2,3,5,7,8,11,13),$$  $$(a_{k,19})_{k=1}^{9} = (1,2,3,4,5,7,9,11,13).$$
\end{thm}
\begin{proof}
The orthogonality of the Dirichlet characters modulo $p$ may be used to
derive
\begin{align}
  \sum_{\chi(-1) = -1} \chi(a)\overline{\chi}(b) =
  \begin{cases}
    \pm \varphi(p)/2, & a \equiv \pm b \pmod{p}, \\
  0, & a \not\equiv \pm b \pmod{p},
  \end{cases}
\end{align}
Therefore, if $\{a_{k,p}\}_{k=1}^{(p-1)/2} = (\Bbb Z / p \Bbb Z)^{*} / \{\pm
1\}$ and $\{\chi_{2s}\}_{s=1}^{(p-1)/2}$ are odd, the rows of
\begin{align}
  (B)_{k,s} = \chi_{2s}(a_{k,p}), \qquad 1 \le k,s \le (p-1)/2,
\end{align}
are orthogonal with respect to the Hermitian inner product. Hence, the matrix
$B$ is an invertible linear transformation corresponding to the change of basis for $ \mathcal{M}_{1}\left ( \Gamma_{1}(p) \right )$
\begin{align}
 B\Bigl ( E_{1,\chi_{2}}(\tau), \ldots,  E_{1,\chi_{2p}}(\tau) \Bigr )^{T} = \Bigl ( \langle a_{1,p} \rangle ( \mathcal{E}_{p} ), \ldots,   \langle a_{(p-1)/2,p} \rangle ( \mathcal{E}_{p} ) \Bigr )^{T}.
\end{align}
 The normalization claims of Theorem \ref{gtd} may be verified from $q$-expansions for the linear combination of Eisenstein series defining
each basis element in the image.  
\end{proof}
\begin{thm} \label{te}
The elements from \eqref{rl} are permuted up to a change of sign by $\Gamma_{0}(p)$ under $\langle \cdot \rangle$, with permutation representation isomorphic to $( \Bbb Z / p \Bbb Z)^{*}/\{\pm 1\}$.
\end{thm}
\begin{proof}
If the set of odd primitive Dirichlet characters modulo $p$ is given
by $\{\chi_{2k}\}_{k=1}^{(p-1)/2}$, the elements of $\mathcal{D} = \{\langle d \rangle
( \mathcal{E}_{p} ) \mid d \in (\Bbb Z / p \Bbb Z)^{*}\}$ are in
bijective correspondence with $$\mathfrak{E} = \{\mathcal{A}(d) := diag(\chi_{2}(d), \chi_{4}(d),\ldots
\chi_{p-1}(d)) \in GL(\frac{p-1}{2},\Bbb C) \mid d\in
(\Bbb Z/p\Bbb Z)^{*}\}.$$
Note that $\mathfrak{E}$ is a group under multiplication corresponding to
action by $\Gamma_{0}(p)$ on $\mathcal{D}$, and $\mathcal{A}(-d)
= - \mathcal{A}(d)$. Therefore, modulo a change of sign, the
permutation representation for the action of $\Gamma_{0}(p)$ on $D$ is isomorphic to the
homomorphic image of $\Gamma_{0}(p)$ in $PGL(\frac{p-1}{2},\Bbb C)$
under $\kappa: \gamma \mapsto
diag(\chi_{2}(\gamma_{22}), \chi_{4}(\gamma_{22}),\ldots
\chi_{p-1}(\gamma_{22}))$, were $\gamma_{22}$ is the lower right-hand
entry of $\gamma$. Since the map $\delta: \gamma \mapsto \gamma_{22} \pmod{p}$ is a surjection from $\Gamma_{0}(p)$ to $(\Bbb Z / p
\Bbb Z)^{*}$ with kernel $\Gamma_{1}(p)$, we conclude $\Gamma_{0}(p)/\Gamma_{1}(p)\cong (\Bbb Z /
p \Bbb Z)^{*}$. Therefore, the projection of the image of $\kappa$ in
$PGL((p-1)/2,\Bbb C)$ is isomorphic to $( \Bbb Z / p \Bbb Z)^{*}/\{\pm 1\}$.
\end{proof}
The proof of Theorem \ref{te} implies that the action of $d\in (\Bbb Z / p \Bbb
Z)^{*}$ on $\mathcal{E}_{p}$ satisfies 
\begin{align}
  \label{eq:71}
\langle d \rangle \langle a_{k,p} \rangle ( \mathcal{E}_{p}) = \langle d\cdot a_{k,p} \rangle (
\mathcal{E}_{p}) =  \pm \langle a_{k,p}' \rangle (
\mathcal{E}_{p}),
\quad \pm d\cdot a_{k,p} \equiv a_{k,p}' \in \{a_{k,p}\}_{k=1}^{(p-1)/2}.
\end{align}
We now show that the normalized Eisenstein sums from Theorem \ref{te} are
synonymous with the products \eqref{ab}--\eqref{iwb} from the Introduction. 
To prove this, we show that each basis
element of level $p$ for $\mathcal{M}_{1}\left ( \Gamma_{1}(p) \right
)$ from Theorem \ref{gtd} is representable
as a quotient of modified theta constants
with  Jacobi
triple product representation \cite[p. 141]{MR1850752}
\begin{align}
  \label{eq:81}
  \theta \left [ {m/n \atop 1} \right ](n \tau) = \exp\left(\frac{
      \pi i m}{2n} \right )
  q^{m^{2}/(8n)} (q^{(n-m)/2}; q^{n})_{\infty}
  (q^{(n+m)/2};q^{n})_{\infty} (q^{n}; q^{n})_{\infty}.
\end{align}
We derive each theta quotient by writing the
product formulations of Theorem \ref{gtd} in terms of modified theta constants
and applying transformations for the theta constants.
\begin{thm} \label{gtd1}
Define $\varphi_{\ell, k}$ by \eqref{thetac}, and, for $[b_{1}, \ldots, b_{(p-1)/2} ] \in \Bbb Z^{(p-1)/2}$, denote
\begin{align}
  \mathfrak{T}_{p} [b_{1}, \ldots, b_{(p-1)/2} ](\tau) =
  \eta^{3}(p\tau) \prod_{k=1}^{(p-1)/2} \exp\left ( -\frac{\pi ib_{k} (2k-1)}{2p} \right )\varphi_{p,k}^{b_{k}}(\tau).
\end{align}
The bases for $\mathcal{M}_{1}(\Gamma_{1}(p))$ from Theorem \ref{gtd} have the theta quotient representations:
\vspace{-0.02in}
  \begin{center}
    \begin{longtable}[h]{||c|c||} \hline \hline 
      Level, $p$ & Basis for $\mathcal{M}_{1}(\Gamma_{1}(p))$ \\ \hline $5$ & $\langle 1 \rangle (\mathcal{E}_{5})= \mathfrak{T}_{5}[2,-3]$, \  $\langle 2 \rangle (\mathcal{E}_{5}) = \mathfrak{T}_{5}[-3,2]$ \\ \hline $7$ & $\langle 1 \rangle (\mathcal{E}_{7})= \mathfrak{T}_{7}[1,0,-2]$, \  $\langle 2 \rangle (\mathcal{E}_{7}) = \mathfrak{T}_{7}[-2,1,0]$, \  $\langle 3 \rangle (\mathcal{E}_{7}) = \mathfrak{T}_{7}[0,-2,1]$  \\ \hline $11$ & $\langle 1 \rangle (\mathcal{E}_{11})= \mathfrak{T}_{11}[0,1,0,-1,-1]$, \ $\langle 2 \rangle (\mathcal{E}_{11})= \mathfrak{T}_{11}[-1,0,0,1,-1]$ \\  \  & $\langle 3 \rangle (\mathcal{E}_{11})= \mathfrak{T}_{11}[1,-1,-1,0,0]$, \  $\langle 5 \rangle (\mathcal{E}_{11})= \mathfrak{T}_{11}[0,-1,1,-1,0]$ \\ & $\langle 7 \rangle (\mathcal{E}_{11})= \mathfrak{T}_{11}[-1,0,-1,0,1]$  \\ \hline $13$ & $\langle 1 \rangle (\mathcal{E}_{13})= \mathfrak{T}_{13}[1,0,0,-1,0,-1]$, \  $\langle 2 \rangle (\mathcal{E}_{13})= \mathfrak{T}_{13}[-1,-1,0,1,0,0]$, \\  \  & $\langle 3 \rangle (\mathcal{E}_{13})= \mathfrak{T}_{13}[0,0,-1,0,1,-1]$, \  $\langle 4 \rangle (\mathcal{E}_{13})= \mathfrak{T}_{11}[0,1,-1,-1,0,0]$, \\  \  & $\langle 5 \rangle (\mathcal{E}_{13})= \mathfrak{T}_{13}[0,-1,1,0,-1,0]$, \  $\langle 7 \rangle (\mathcal{E}_{13})= \mathfrak{T}_{13}[-1,0,0,0,-1,1]$  \\ \hline $17$ & $\langle 1 \rangle (\mathcal{E}_{17})= \mathfrak{T}_{17}[1,0,0,0,0,-1,-1,0]$, \\ \ &  $\langle 2 \rangle (\mathcal{E}_{17})= \mathfrak{T}_{17}[0,-1,0,0,1,0,0,-1]$, \\  \  & $\langle 3 \rangle (\mathcal{E}_{17})= \mathfrak{T}_{17}[0,0,0,-1,0,1,0,-1]$, \\ \ &  $\langle 5 \rangle (\mathcal{E}_{17})= \mathfrak{T}_{17}[0,0,0,1,-1,-1,0,0]$, \\  \  & $\langle 7 \rangle (\mathcal{E}_{17})= \mathfrak{T}_{17}[0,-1,1,0,0,0,-1,0]$, \\ \ &  $\langle 8 \rangle (\mathcal{E}_{17})= \mathfrak{T}_{17}[0,0,-1,0,-1,0,0,1]$ \\  \  & $\langle 11 \rangle (\mathcal{E}_{17})= \mathfrak{T}_{17}[-1,1,-1,0,0,0,0,0]$, \\ \ &  $\langle 13 \rangle (\mathcal{E}_{17})= \mathfrak{T}_{17}[-1,0,0,-1,0,0,1,0]$,   \\ \hline $19$ & $\langle 1 \rangle (\mathcal{E}_{19})= \mathfrak{T}_{19}[1,1,0,0,-1,-1,-1,0,0]$, \\ & $\langle 2 \rangle (\mathcal{E}_{19})= \mathfrak{T}_{19}[0,-1,-1,0,1,1,0,-1,0]$, \\  \  & $\langle 3 \rangle (\mathcal{E}_{19})= \mathfrak{T}_{19}[1,-1,0,0,-1,0,1,0,-1]$, \\ &  $\langle 4 \rangle (\mathcal{E}_{19})= \mathfrak{T}_{19}[0,0,1,-1,0,-1,0,1,-1]$, \\  \  & $\langle 5 \rangle (\mathcal{E}_{19})= \mathfrak{T}_{19}[0,0,-1,1,0,0,-1,1,-1]$, \\ &  $\langle 7 \rangle (\mathcal{E}_{19})= \mathfrak{T}_{19}[0,0,1,-1,-1,1,0,-1,0]$  \\  \  & $\langle 9 \rangle (\mathcal{E}_{19})= \mathfrak{T}_{19}[-1,-1,0,-1,0,0,1,0,1]$, \\ &  $\langle 11 \rangle (\mathcal{E}_{19})= \mathfrak{T}_{19}[-1,0,0,1,0,0,-1,-1,1]$ \\   &  $\langle 13 \rangle (\mathcal{E}_{19})= \mathfrak{T}_{19}[-1,1,-1,0,1,-1,0,0,0]$ \\ \hline \hline
\end{longtable}
  \end{center}
\end{thm}
\vspace{-0.5in}
\begin{proof}
The theta quotient representations for $\mathcal{E}_{p}(\tau) = \langle 1 \rangle (\mathcal{E}_{p})(\tau)$ may be deduced from the product representations proved in Theorem \ref{gtd}.
Transformation formula for these theta quotients under
$\Gamma_{0}(p)$, in turn, may be deduced from corresponding modular
transformation formulas for $\eta(\tau)$ from \cite[p. 51]{knopp} and those
for vectors of modified theta constants $\mathcal{V}_{N}(\tau)$ under
generators for the full modular group from Theorem \ref{far}. For each
prime $p$, we may deduce the product representations for each
normalized
Eisenstein sum $\langle a_{k,p} \rangle (\mathcal{E}_{p})$ from the
modular transformation formulas for these building blocks.
We illustrate the general procedure with $p = 5$. From
Theorem \ref{sf} and \eqref{eq:81}, 
\begin{align}
  \label{eq:99}
  \langle 1 \rangle (\mathcal{E}_{5}) &=   \frac{(q;q)_{\infty}^{2}}{(q, q^{4};
    q^{5})_{\infty}^{5}} = \eta^{3}(5 \tau) \frac{e^{- 2 \pi
      i/10}\varphi_{5,1}^{2}(\tau)}{e^{- 9 \pi
      i/10}\varphi_{5,2}^{3}(\tau)} = \mathfrak{T}_{5}[2,-3](\tau). 
\end{align}
A set of generators for $\Gamma_{0}(5)$ is given by
\begin{align}
  \label{eq:93}
  T = \begin{pmatrix}
    1 & 1 \\ 0 & 1
  \end{pmatrix}, \quad
  \alpha = \begin{pmatrix}
    2 & -1 \\ 5 & -2
  \end{pmatrix}, \quad 
  \beta = \begin{pmatrix}
    3 & -2 \\ 5 & -3
\end{pmatrix}.
\end{align}
We now employ transformation formulas up to a constant multiple for the weight $1/2$ vector of modified
theta constants $[ \varphi_{5,2}, \varphi_{5,1}]^{tr}$. 
We begin with parameterizations for the generators of $\Gamma_{0}(5)$ in terms
of those for the full modular group
\begin{align}
  \label{eq:58}
\alpha = TST^{2}ST^{3}S&, \qquad \beta =   TST^{3}ST^{2}S. 
\end{align}
Transformation matrices for the vectors
 of modified theta constants may be computed from their images in
 $PGL((p-1)/2,\Bbb C)$ via the representation $\pi_{p}$ given in \eqref{eq:31}--\eqref{eq:31a}
\begin{align}
  \label{eq:57}
  \pi_{5}(T) &= \begin{pmatrix}
    e^{\frac{9 \pi i}{20}} & 0 \\ 0 & e^{\frac{\pi i}{20}}
  \end{pmatrix}, \ \  \pi_{5}(\alpha) =  \left ( \begin{array}{cc}
 0 & e^{\frac{\pi i}{20}} \\
 -e^{\frac{9\pi i}{20}} & 0 \\
\end{array}
\right), \ \  \pi_{5}(\beta) = \left(
\begin{array}{cc}
 0 & e^{\frac{\pi i}{4}} \\
 -e^{\frac{\pi i}{4}} & 0 \\
\end{array}
\right).
\end{align}
Hence, by \eqref{eq:99}, and the modular transformation formula for
$\eta(\tau)$, we deduce that up to a constant multiple, $C$, 
\begin{align}
  \label{eq:101}
\langle 2 \rangle (\mathcal{E}_{5}) = (5 \tau - 3)^{-1}\mathfrak{T}_{5}[2,-3](\beta
\tau) = C\eta^{3}(5 \tau)
\frac{\varphi_{5,2}^{2}(\tau)}{\varphi_{5,1}^{3}(\tau)} =
C \frac{e^{6\pi i/10}}{e^{3 \pi i/10}}q + O(q^{2}).
\end{align}
On the other hand, from transformation formulas satisfied
by $E_{\chi_{2,5},1}(\tau)$ and $E_{\chi_{4,5},1}(\tau)$,
\begin{align}
  \label{eq:102}
\langle 2 \rangle (\mathcal{E}_{5}) =  \mathcal{E}_{5}(\beta \tau) 
= \chi_{2,5}(2)  E_{\chi_{2,5},1}( \tau) + \chi_{4,5}(2)
E_{\chi_{4,5},1}( \tau) =q + O(q^{2}).
\end{align}
Therefore, $C = e^{-3 \pi i/10}$, and so
\begin{align}
  \label{eq:103}
  \langle 2 \rangle (\mathcal{E}_{5}) = e^{-3 \pi i/10} \eta^{3}(5 \tau)
\frac{\varphi_{5,2}^{2}(\tau)}{\varphi_{5,1}^{3}(\tau)} = \mathfrak{T}_{5}[-3,2](
\tau) = q\frac{(q;q)_{\infty}^{2}}{(q^{2}, q^{3}; q^{5})_{\infty}^{5}}.
\end{align}

For higher levels $7 \le p \le 19$, we similarly use the fact that the image
of $\Gamma_{0}(p)$ under the presentation $\pi_{p}$  defined by Theorem \ref{far} is a matrix with a single nonzero
entry in each row and column. We obtain the theta quotient
representations of the bases for $\mathcal{M}_{1}(\Gamma_{1}(p))$
from those for $\mathcal{E}_{p}(\tau)$. In each case, we permute the theta quotients according to the image of $\pi_{N}$
 and apply the transformation formulas for Eisenstein series on $\Gamma_{0}(p)$,
$(\gamma_{21}\tau + \gamma_{22})^{-k}E_{k,\chi}(\gamma \tau) =
\chi(\gamma_{22})E_{k,\chi}(\tau)$ for $\gamma \in \Gamma_{0}(p)$, to
each componenent of $\mathcal{E}_{p}(\tau)$. We then compare
the first nonzero entry in the resulting $q$-expansions. By repeating
this process with an independent set of generators for
$\Gamma_{0}(p)$, we ultimately obtain the linearly independent sets of
theta quotient representations claimed in Theorem \ref{gtd1}.
\end{proof}



We next extend the permuted bases of weight one forms for $\Gamma_{1}(p)$ from
Theorems \ref{gtd} and \ref{gtd1} to generators for the graded algebra
of positive integer weight modular forms for $\Gamma_{1}(p)$. By Lemma \ref{lem:nm}, it
suffices to  prove that monomials of degree $k=2,3$ in the prospective
weight one
generators span the vector space $\mathcal{M}_{k}(\Gamma_{1}(p))$.
Lemma \ref{lem:mod}  demonstrates the existence of  $\dim
\mathcal{M}_{k}(\Gamma_{1}(p))$ linearly independent monomials of degree $k$ in the generators. The dimensions of
$\mathcal{M}_{k}(\Gamma_{1}(p))$ for $k = 2,3$ are given by \cite{dish}:
\begin{center}
\begin{tabular}{||l|c|c|c|c|c|c||}
\hline \hline   \ & $p = 5$ & $p = 7$ & $p=11$ & $p=13$ & $p=17$ & $p=19$ \\ \hline 
$\dim \mathcal{M}_{2}(\Gamma_{1}(p))$ & $3$ & $5$ & $10$ & $13$ & $20$ & $24$  \\ \hline
$\dim \mathcal{M}_{3}(\Gamma_{1}(p))$ & $4$ & $7$ & $15$ & $20$ & $32$ & $39$ \\ \hline \hline
\end{tabular}
\end{center}

\begin{thm} \label{lem:mod} Let $\langle a_{k,p} \rangle ( \mathcal{E}_{p})(\tau)$ be defined as in Theorems \ref{gtd} and \ref{gtd1}. For each prime $5\le p\le 19$, the set $\{\langle a_{k,p} \rangle ( \mathcal{E}_{p})(\tau)\}_{k-1}^{(p-1)/2}$ generates the graded ring $\mathcal{M}(\Gamma_{1}(p))$. 
\end{thm}
    \begin{proof}
Denote the image of the normalized Eisenstein sum under $a \in (\Bbb Z
/ n \Bbb Z)^{*}$ by
\begin{align} \label{nie}
\mathcal{E}_{a,p}(\tau) := \langle a \rangle ( \mathcal{E}_{p})(\tau).  
\end{align}
Since, by \eqref{eq:99} and \eqref{eq:103},  
\begin{align}
  \mathcal{E}_{1,5}(\tau) = 1 +O(q), \qquad \mathcal{E}_{2,5} = q + O(q^{2}),
\end{align}
any set of distinct monomials of degree $k$ forms a linearly independent set of modular forms of weight $k$ for $\Gamma_{1}(5)$.  
Hence, bases for $\mathcal{M}_{k}(\Gamma_{1}(5))$, $k=2,3$, respectively, are
\begin{align}
  \{ \mathcal{E}_{1,5}^{2},\mathcal{E}_{1,5} \mathcal{E}_{2,5},
  \mathcal{E}_{2,5}^{2}  \}, \quad  \{ \mathcal{E}_{1,5}^{3},
  \mathcal{E}_{1,5}^{2} \mathcal{E}_{2,5}, \mathcal{E}_{1,5} \mathcal{E}_{2,5}^{2}, \mathcal{E}_{2,5}^{3}  \}.
\end{align}
For the higher levels $p = 7,11,13,17,19$, we introduce the
complete homogeneous symmetric polynomial in $n$ variables of
degree $k$, 
\begin{align}
  \label{eq:105}
h_{k}(\vec{x}) = h_{k}(x_{1}, x_{2},
\ldots, x_{n}) = \sum_{1 \le i_{1} \le i_{2} \le \cdots \le i_{k} \le
  n} x_{i_{1}} x_{i_{2}} \ldots x_{i_{k}}.  
\end{align}
For each fixed $p,k$, and  $1 \le i,j\le
  {k+\frac{p-1}{2} -1 \choose k}$, let $\mathcal{H}(i,j)$
denote the coefficient of $q^{j}$ in the $q$-expansion of the $i$th monomial
  of 
  \begin{align}
    \label{eq:106}
h_{k} \Bigl (\mathcal{E}_{a_{1,p}}, \mathcal{E}_{a_{2,p}}, \ldots,
  \mathcal{E}_{a_{\frac{p-1}{2}, p}})     
  \end{align}
under lexicographic ordering
  of the monomials of $h_{k}(\vec{x})$. Define the matrix $H_{p,k} =
\{\mathcal{H}(i,j) \}$. For each prime $5\le p \le 19$, a computer algebra
system may be used to show
  \begin{align}
    \label{eq:83}
    Rank(H_{p,k}) = \dim \mathcal{M}_{k}(\Gamma_{1}(p)), \qquad k =
    1,2, 3.
  \end{align}
This proves that a subset of the terms in the polynomial \eqref{eq:106} span
$\mathcal{M}_{k}(\Gamma_{1}(p)$ for $k = 2,3$. 
The proof of Theorem \ref{lem:mod} may be
completed by applying Lemma \ref{lem:nm}.
\end{proof}
For large values of $p$, the verification of \eqref{eq:83} is
nontrivial. In particular, the case $p = 19$, $k = 3$ requires knowledge of the first
$165$ terms in the $q$-expansions of the  $165$ monomials of degree
$3$ in the
parameters from \eqref{eq:85}--\eqref{iwb}. This calculation and the rank
of the corresponding matrix was accomplished in less than $24$
hours on a CAS. For each $p$, an explicit basis for $\mathcal{M}_{k}(\Gamma_{1}(p))$ may be constructed
by row reducing $H_{k,p}$ and selecting the $ \dim \mathcal{M}_{k}(\Gamma_{1}(p))$ monomials corresponding to linear independent rows.

\section{Symmetric representations for modular forms of level $p$} \label{s_eis}

This section is devoted to applications of the symmetric representations for
modular forms in terms of
the permuted generators for $\mathcal{M}(\Gamma_{1}(p))$. Theorem \ref{P} presents a uniform parameterization for an important
class of combinatorial generating functions, that of $t$-cores.
For a given partition $\lambda$, each square in the Young diagram
representation for $\lambda$ defines a \textit{hook} consisting of
that square, all the squares to the right of that square, and all the
squares below that square. The \textit{hook number} of a given square
is the number of squares in the that hook. A partition $\lambda$ is
said to be a t-core if it has no hook numbers that are multiples of t.
The generating function for the number of $t$-cores is \cite{MR1055707}
\begin{align}
  \label{eq:77}
  \sum_{n=0}^{\infty} c_{t}(n) q^{n} = \frac{(q^{tn};q^{tn})_{\infty}^{t}}{(q;q)_{\infty}}.
\end{align}
Series from the first three cases of Theorem
\ref{P} play a fundamental role in the derivation of Ramanujan's
famous congruences for the partition function modulo $5, 7, 11$
\cite{MR1701582,MR2280868}. 
\begin{thm} \label{P} Let $a_{k,p}$ be defined by Theorem \ref{gtd},
  $\mathcal{E}_{a,p}$ by \eqref{nie}, 
  and let $\delta_{p} = \frac{p^{2} -1}{24}$. Then, for each prime $p$ with $5 \le p \le 19$, 
\begin{align}
  \label{eq:73}
  q^{\delta_{p}} \frac{(q^{p}; q^{p})_{\infty}^{p}}{(q;q)_{\infty}} =
  \prod_{k=1}^{\frac{p-1}{2}} \mathcal{E}_{a_{k},p}(\tau).
\end{align}
\end{thm}
\begin{proof}
  The claim \eqref{eq:73} for each level $p$ may be deduced by
  utilizing the product representations, given explicitly by \eqref{ab}--\eqref{iwb}, for the generators from Theorem \ref{gtd1}.
\end{proof}

From Theorem \ref{P}, we deduce new divisor sum representations for
$p$-cores, $p$ prime, $5 \le p \le 19$, in terms of $L$-function
values for odd Dirichlet characters at the origin.
\begin{cor} \label{cortc}
 Let $\delta_{p}$ be defined by Theorem \ref{P}. For primes $5 \le p
 \le 19$, and $n \ge \delta_{p}$,
  \begin{align}
    \label{eq:107}
    c_{p}(n- \delta_{p}) = \left ( \frac{2}{p-1} \right )^{\frac{p-3}{2}}&\sum_{r_{1}+\cdots+r_{(p-3)/2}=n}\left(\prod_{k=1}^{\frac{p-3}{2}}\sum_{d_{k}\mid
        r_{k}}\ell_{p}(a_{k+1,p}  \cdot d_{k})\right) \\ &+\left ( \frac{2}{p-1} \right )^{\frac{p-1}{2}}\sum_{ r_{1} + \cdots + r_{(p-1)/2} =
      n} \left ( \prod_{k=1}^{\frac{p-1}{2}} \sum_{d_{k} \mid r_{k}}
    \ell_{p,a_{k,p}}(a_{k,p} \cdot d_{k}) \right ),  \nonumber
  \end{align}
where, $r_{i} \ge 1$, $a_{k,p}$ is defined by Theorem \ref{gtd}, and
for $d \in \Bbb Z$,  we define
\begin{align}
  \label{eq:108}
 \ell_{p}(d) = 2\sum_{{\chi(-1) = -1 \atop primitive\ mod\ p}} \frac{\chi(d)}{L(0, \chi)}.
\end{align}
\end{cor}
\begin{proof}
Identity \eqref{eq:107} may be derived from
  \eqref{eq:77}, \eqref{eq:73} and since, for $1 \ne a\in (\Bbb Z / p \Bbb
  Z)^{*}$,
  \begin{align}
    \label{eq:109}
    \mathcal{E}_{a,p}(\tau) = \frac{2}{p-1} \sum_{\chi(-1)= -1}
    \chi(a)E_{1,\chi}(\tau) = \frac{2}{p-1}\sum_{n=1}^{\infty} \left (
      \sum_{d \mid n} \sum_{\chi(-1) = -1} \frac{
        2\chi(ad)}{L(0,\chi)} \right )q^{n}.
  \end{align}
The right side of \eqref{eq:107} is the
Cauchy product of coefficients of the series from \eqref{eq:73}.
\end{proof}
For the primes $p = 5,7,13$, 
the Eisenstein series of weight two and trivial character modulo $p$ have common
representations in terms of the squares of the generators. 
\begin{thm} If $\chi_{1,p}$ is the principal character modulo
  $p$, and $\mathcal{E}_{a,p}$ is given by \eqref{nie}, 
  \begin{align}
    \label{eq:78} 
    1 + 6 &\sum_{n=1}^{\infty} \frac{\chi_{1,5}(n) n q^{n}}{1 - q^{n}}
    = \mathcal{E}_{1,5}^{2} + \mathcal{E}_{2,5}^{2}, \ \    
1 + 4 \sum_{n=1}^{\infty} \frac{\chi_{1,7}(n) n q^{n}}{1 - q^{n}}
    = \mathcal{E}_{1,7}^{2} +
    \mathcal{E}_{2,7}^{2}+\mathcal{E}_{3,7}^{2}, \\ 
&1 + 2 \sum_{n=1}^{\infty} \frac{\chi_{1,13}(n) n q^{n}}{1 - q^{n}}
    = \mathcal{E}_{1,13}^{2} +
    \mathcal{E}_{2,13}^{2}+\mathcal{E}_{3,13}^{2}
    +\mathcal{E}_{4,13}^{2} +
    \mathcal{E}_{5,13}^{2}+\mathcal{E}_{7,13}^{2}.   \label{eq:78b} 
  \end{align}
\end{thm}

We next turn our attention to the common coefficients within the differential systems from
Theorems \ref{th:11} and \ref{th:12}. We will apply the following
standard result with $h = pk/12$.
\begin{lem} \cite[Lemma 2.1]{MR2024734}  \label{afn}
  Suppose that $p$ is a prime and that $f$ is a meromorphic modular
  form of weight $k$ on $\Gamma_{0}(p)$. If $h$ is any constant, then
  the function
  \begin{align}
    \label{eq:110}
    F_{f} := q \frac{d}{dq} f + \left \{ \frac{\frac{k}{12} - h}{p-1}
      \cdot p E_{2}(p \tau)) + \frac{h - \frac{pk}{12}}{p-1} \cdot
      E_{2}(\tau) \right \} \cdot f \in \mathcal{M}_{k+2}(\Gamma_{0}(p)).
  \end{align}
\end{lem}
The proof of Lemma \ref{afn} from \cite{MR2024734} shows that the
Lemma remains true for the subgroup $\Gamma_{1}(p)$ of
$\Gamma_{0}(p)$. Moreover, if $f, g \in
\mathcal{M}_{k}(\Gamma_{1}(p))$, and $\gamma \in \Gamma_{0}(p)$ with $\langle \gamma \rangle (f) = g$, then
$\langle \gamma \rangle (F_{f}) = F_{g}$. Thus, for $a \in (\Bbb Z / p
\Bbb Z)^{*}$, $( q \frac{d}{dq}
\mathcal{E}_{a,p}(\tau) )/\mathcal{E}_{a,p}(\tau)-
\frac{pE_{2}(p\tau)}{12} \in \mathcal{M}_{2}(\Gamma_{1}(p))$.
Relevant basis expansions are encoded by Theorem
\ref{bu}. These subsume the systems of level $5$ and $7$
of Theorems \ref{th:11}--\ref{th:12} and new coupled systems of
level $11\le p \le 19$. The differential equations are permuted by
$\Gamma_{0}(p)$, with invariance subgroup $(\Bbb Z / p \Bbb
Z)^{*}/\{\pm 1\}$. 
\begin{thm} \label{bu} Let $\mathcal{E}_{a,p}$ be given as in
  \eqref{nie} and $a_{k,p}$ as in Theorem \ref{gtd}, and define 
\begin{align}
  \mathcal{F}_{5}(x_{1}, x_{2}) &= - 5 x_{1}^{2} +66x_{1}x_{2}+ 7 x_{2}^{2},\\ \mathcal{F}_{7}(x_{1}, x_{2},x_{3}) &= - 7 x_{1}^{2} +5x_{2}^{2}+ 5 x_{3}^{2}-20x_{2}x_{3}+52x_{1}x_{2}, \\ 
\mathcal{F}_{11}(x_{1}, x_{2},\cdots, x_{5}) &=
-11 x_{1}^{2}+x_{2}^{2}+13
x_{3}^{2}+x_{4}^{2}+x_{5}^{2} \\ &+34 x_{3} x_{1} +42 x_{2} x_{4}-40 x_{2} x_{5}+38
x_{3} x_{5}-10 x_{4} x_{5}, \nonumber \\ 
\mathcal{F}_{13}(x_{1}, x_{2}, \cdots ,x_{6}) &=
-13 x_{1}^2+11
x_{2}^{2}-x_{3}^{2}-x_{4}^{2}-x_{5}^{2}+11 x_{6}^{2}+16 x_{2}x_{1}  \\
& +38 x_{1} x_{5}+2 x_{2} x_{3}-20
x_{2} x_{4}+40 x_{2}x_{5}-8 x_{2} x_{6}+14 x_{4} x_{6}, \nonumber \\ 
\mathcal{F}_{17}(x_{1}, x_{2}, \ldots, x_{8}) &=
-17 x_1^2+19 x_2^2+7 x_3^2-5 x_4^2-5 x_5^2-5 x_6^2-5 x_7^2+19 x_8^2
\\ &-12
x_2 x_3+54 x_2 x_5+12 x_5 x_6+30 x_2 x_7-42 x_5 x_7  \nonumber \\ &-60 x_2 x_8+12 x_3
x_8-54 x_5 x_8-12 x_7 x_8, \nonumber \\
 \mathcal{F}_{19}(x_{1}, x_{2}, \ldots, x_{9}) &=
-19 x_1^2+41 x_2^2+5 x_3^2-7 x_4^2-7 x_5^2+5 x_6^2+5 x_7^2-7 x_8^2-7
x_9^2 \nonumber \\ &+16 x_2 x_5-36 x_3 x_5+8 x_4 x_5+32 x_5 x_6+12 x_4
x_7-16 x_5 x_7 \\ &+28 x_5 x_8-12 x_6 x_8-4 x_2 x_9+68 x_3 x_9+16 x_5
x_9-28 x_8 x_9. \nonumber 
\end{align}
Then, the generators satisfy a system of $(p-1)/2$ differential equations subsumed by
  \begin{align} \label{some}
\frac{12}{\mathcal{E}_{a,p}} q \frac{d}{dq} \mathcal{E}_{a,p} &=
\mathcal{F}_{p}(\mathcal{E}_{a\cdot a_{1},p}, \mathcal{E}_{a\cdot a_{2},p},\ldots \mathcal{E}_{a\cdot a_{(p-1)/2},p}) +pE_{2}(p\tau), \quad  a \in (\Bbb Z / p \Bbb Z)^{*}.
\end{align}
\end{thm}
%
A differential system for a full set of generators for
$\mathcal{M}(\Gamma_{1}(p))$ may be formulated from these
relations \eqref{some}, and the group action formulation \eqref{eq:71}. For $p = 11$,
\begin{align}
\frac{12}{\mathcal{E}_{1,11}} q \frac{d}{dq} \mathcal{E}_{2,11} &=
\mathcal{F}_{11}(\mathcal{E}_{1,11}, \mathcal{E}_{2,11},\mathcal{E}_{3,11}, \mathcal{E}_{5,11}, \mathcal{E}_{7,11}) +11
P\left(q^{11}\right ), \\ 
\frac{12}{\mathcal{E}_{2,11}} q \frac{d}{dq} \mathcal{E}_{2,11} &=
\mathcal{F}_{11}(\mathcal{E}_{2,11}, -\mathcal{E}_{7,11},
-\mathcal{E}_{5,11}, -\mathcal{E}_{1,11}, \mathcal{E}_{3,11}) +11
P\left(q^{11}\right), \\  
\frac{12}{\mathcal{E}_{3,11}} q \frac{d}{dq} \mathcal{E}_{3,11} &=
\mathcal{F}_{11}(\mathcal{E}_{3,11}, -\mathcal{E}_{5,11},
-\mathcal{E}_{2,11}, -\mathcal{E}_{7,11}, -\mathcal{E}_{1,11}) +11 P\left(q^{11}\right),
\\  
\frac{12}{\mathcal{E}_{5,11}} q \frac{d}{dq} \mathcal{E}_{5,11} &=
\mathcal{F}_{11}(\mathcal{E}_{5,11}, -\mathcal{E}_{1,11},
-\mathcal{E}_{7,11}, \mathcal{E}_{3,11}, \mathcal{E}_{2,11}) +11
P\left(q^{11}\right),
\\  
\frac{12}{\mathcal{E}_{7,11}} q \frac{d}{dq} \mathcal{E}_{7,11} &=
\mathcal{F}_{11}(\mathcal{E}_{7,11}, \mathcal{E}_{3,11},
-\mathcal{E}_{1,11}, \mathcal{E}_{2,11}, \mathcal{E}_{5,11}) +11 P\left(q^{11}\right).
\end{align}

\section{Quadratic relations and Klein's automorphism groups} \label{s_higher}
The following results relate permuted generators for
$\mathcal{M}(\Gamma_{1}(p))$ for primes $5 \le p \le 19$ to Klein's
classical automorphism groups and extensions. We begin with symmetries of the icosahedron. Define the normalized theta constant of order $p$ and index $k$ by
\begin{align}
  \phi_{p,k}(\tau) = \exp \left ( - \frac{(2k-1)\pi i}{2p} \right ) \varphi_{p,k}(\tau).
\end{align}
To study the geometry of $X(N)$, the extended upper half plane modulo the
principal congruence subgroup $\Gamma(N)$, Klein devised vector-valued
modular forms and used them to formulate explicit maps into the
relevant moduli spaces. At level $p = 5$, Klein considered
$V_{5}(\tau) = [\phi_{5,2}(\tau), \phi_{5,1}(\tau)]^{T}$ and showed
$v_{5}$ satisfies transformation formulas under generators for the full modular group
\begin{align}
 V_{5}(S\tau) := V(-1/\tau) =  \tau^{1/2}\rho_{5}(S)V_{5}(\tau),\quad
 V_{5}(T\tau): = V_{5}(\tau +1)= \rho_{5}(T)V_{5}(\tau)
\end{align}
Klein showed that $\rho_{5}:PSL(2, \Bbb Z) \to PLG(2, \Bbb C)$ is a
representation for the automorphism group of the projected
icosahedron. Namely, $\langle \rho_{5}(S), \rho_{5}(T)\rangle = PSL(2,
\Bbb F_{5}) = A_{5}$. Our symmetric representations for modular forms
in terms of $\mathcal{E}_{1,5}, \mathcal{E}_{2,5}$ make use of the
fact that, up to a constant multiple, $\rho_{5}\! \! \mid_{\Gamma_{0}(5)}$ induces a permutation of $\phi_{5,1}, \phi_{5,2}$. A similar analysis at level $p=7$ gives rise to the Klein quartic curve defined by the locus
\begin{align} \label{wya}
  a^{3}b + b^{3} c+c^{3}a = 0, \qquad [a:b:c] \in \Bbb C \Bbb P^{2}.
\end{align}
At levels $p=7,11$, certain quadratic relations between the $\Gamma_{0}(p)$-permuted generators for $\mathcal{M}(\Gamma_{1}(p))$ define the relevant curves in complex projective space isomorphic to those defining the modular curve $X(p)$. At level seven, apply Theorem \ref{gtd1} to observe that
\begin{align}
  \mathcal{E}_{1,7} = \frac{\phi_{7,1}}{\phi_{7,3}^{2}},\quad  \mathcal{E}_{2,7} = \frac{\phi_{7,2}}{\phi_{7,1}^{2}},\quad  \mathcal{E}_{3,7} = \frac{\phi_{7,3}}{\phi_{7,2}^{2}} = - \mathcal{E}_{4,7}.
\end{align}
With $a = \phi_{7,1}, b = \phi_{7,2}, c = - \phi_{7,3}$, Theorem \ref{eqt} provides a parameterization for \eqref{wya}. 
\begin{thm} \label{eqt} Define $\mathcal{E}_{a,p} =
  \mathcal{E}_{a,p}(\tau)$ as in
  \eqref{nie}. Then 
  \begin{align}
    \label{eq:10}
     \mathcal{E}_{2,7}  \mathcal{E}_{4,7} +  \mathcal{E}_{2,7} \mathcal{E}_{1,7}+  \mathcal{E}_{4,7} \mathcal{E}_{1,7}= 0. 
  \end{align}
\end{thm}
\begin{proof}

Replace $q$ by $q^{7}$ in \eqref{eq:13}, and make the respective
substitutions 
\begin{align}
  \label{eq:133}
  (x, y) = (\pi \tau, 2 \pi \tau), \quad (\pi \tau, 3 \pi \tau), \quad
  (2 \pi \tau, 3 \pi \tau)
\end{align}
in the resultant identities to derive, from \eqref{another},
\eqref{jtp}, and \eqref{yi1}, 
\begin{align} \label{kg} 
  D_{2}(q) - D_{1}(q) &=  \mathcal{E}_{3,7} \mathcal{E}_{1,7},  \quad D_{1}(q) - D_{3}(q) = \mathcal{E}_{2,7}
  \mathcal{E}_{1,7}, \\   & D_{3}(q) - D_{2}(q) =  \mathcal{E}_{2,7}  \mathcal{E}_{3,7}, \label{kg1}
\end{align}
where $D_{1}(q)$, $D_{2}(q)$, and $D_{3}(q)$ take the form
\begin{align} \label{rn}
  D_{1}(q) = \sum_{n=1}^{\infty} \frac{n q^{n}}{1 - q^{7n}} +&
  \sum_{n=1}^{\infty} \frac{nq^{6n}}{1 - q^{7n}},\ \ D_{2}(q) =  \sum_{n=1}^{\infty} \frac{n q^{2n}}{1 - q^{7n}} + \sum_{n=1}^{\infty} \frac{nq^{5n}}{1 - q^{7n}}, \\ 
 &D_{3}(q) =  \sum_{n=1}^{\infty} \frac{n q^{3n}}{1 - q^{7n}} + \sum_{n=1}^{\infty} \frac{nq^{4n}}{1 - q^{7n}}. \label{rn1}
\end{align}
Identity \eqref{eq:10} follows immediately from \eqref{kg}--\eqref{kg1}.
\end{proof}

The level $11$ analogue of the Klein quartic is the Klein cubic
\cite{MR522700,klein_el}, \cite[Band II]{MR0183872} 
\begin{align}
  \label{eq:64}
  v^{2} w + w^{2} x + x^{2} y + y^{2} z + z^{2} v = 0, \qquad [v,w,x,y,z] \in \Bbb C \Bbb P^{4}.
\end{align}
The curve $X(11)$ is isomorphic to the locus in $\Bbb C \Bbb P^{4}$ consisting of $[v,w,x,y,z]$ with \cite{MR1416724,MR1510086}
\begin{align}
  \label{eq:112}
  Rank \begin{pmatrix}
    w & v & 0 & 0 & z \\ v & x & w & 0 & 0 \\ 0 & w & y & x & 0 \\ 0 &
    0 & x & z & y \\ z & 0 & 0 & y & v
  \end{pmatrix} = 3.
\end{align}
The following theorem features relations between the level eleven generators
equivalent to theta function parameterizations for the
second and third minors along the top row. 
\begin{thm} \label{ghl} Let $\mathcal{E}_{a,p} =
  \mathcal{E}_{a,p}(\tau) $ be defined as in
  \eqref{nie}. Then 
  \begin{align}
    \label{eq:111}
   \mathcal{E}_{3,11}\mathcal{E}_{7,11}  +
   \mathcal{E}_{1,11}\mathcal{E}_{7,11} -
   \mathcal{E}_{3,11}\mathcal{E}_{5,11} -
   \mathcal{E}_{5,11}\mathcal{E}_{7,11} = 0, \\
 \mathcal{E}_{1,11}\mathcal{E}_{3,11}
 -\mathcal{E}_{1,11}\mathcal{E}_{5,11} +
 \mathcal{E}_{5,11}\mathcal{E}_{7,11}  = 0.    \label{eq:111a}
   \end{align}
\end{thm}
\begin{proof}
  The claimed identities may be derived directly as special cases of
  \eqref{eq:13} or from Theorem \ref{gtd} and the $q$-expansions of the generators $\mathcal{E}_{i,11}$ for $\mathcal{M}(\Gamma_{1}(11))$.
\end{proof}
%
%
From Theorem \ref{gtd1}, \eqref{eq:111}, and \eqref{eq:111a}, we deduce the
quartic theta constant identities
\begin{align}
  \label{eq:119}
  \phi_{11,3} \phi_{11,2}^3-\phi_{11,3}^2 \phi_{11,5}
  \phi_{11,2}+\phi_{11,1} \phi_{11,4} \phi_{11,5} \phi_{11,2}-\phi_{11,1}^2
  \phi_{11,3}^2 = 0, \\ 
\phi_{11,2} \phi_{11,4} \phi_{11,1}^2-\phi_{11,1} \phi_{11,2}
\phi_{11,3}^2+\phi_{11,3} \phi_{11,4} \phi_{11,5}^2 = 0.
\end{align}
These are theta function parameterizations for the zeros of the second and third minors along the
top row of the matrix from \eqref{eq:112} 
\begin{align}
  \label{eq:120}
v^2 x^2-v^2 y z+v y^3+w x y z, \quad -v x^3+v x y z-x y^3.
\end{align}
Under action by $(\Bbb Z /
11 \Bbb Z)^{*}$ via $\langle \cdot \rangle$ and \eqref{eq:71}, identities
\eqref{eq:111}--\eqref{eq:111a} are equivalent to $10$ theta
function identities providing a parameterization for $X(11)$ 
(see \cite[p. 10]{MR1416724}). 

At higher levels, the geometry of the modular curves $X(p)$ was not
examined as closely by Klein, though these cases have been the
subject of \cite{MR1144665,nin,yangL}. The ensuing quadratic
relations may be relevant in describing the modular
curves at levels $13, 17, 19$. We give only a subset of all quadratic relations
and point out that many more may be derived by acting on each identity
by $\Gamma_{0}(p)$. The identities may be proved from the Sturm bound for
each subgroup $\Gamma_{1}(p)$ \cite[\S 3.3]{MR2289048} and the
$q$-expansions of the generators.

\vspace{0.1in}
\begin{thm} If $\mathcal{E}_{a,p}$ is defined by
  \eqref{nie}, the following level $13$ relations hold:
  \begin{align}
\mathcal{E}_{1,13}\mathcal{E}_{3,13} + \mathcal{E}_{2,13}\mathcal{E}_{3,13} - \mathcal{E}_{1,13}\mathcal{E}_{5,13} - \mathcal{E}_{1,13}\mathcal{E}_{7,13} - \mathcal{E}_{2,13}\mathcal{E}_{7,13} - \mathcal{E}_{5,13}\mathcal{E}_{7,13} =0, \\
  \mathcal{E}_{1,13}\mathcal{E}_{4,13} - \mathcal{E}_{1,13}\mathcal{E}_{5,13} +   \mathcal{E}_{3,13}\mathcal{E}_{7,13} +   \mathcal{E}_{4,13}\mathcal{E}_{7,13} -   \mathcal{E}_{5,13}\mathcal{E}_{7,13} =0, \\
\mathcal{E}_{3,13}\mathcal{E}_{5,13}-\mathcal{E}_{1,13}\mathcal{E}_{7,13} - \mathcal{E}_{3,13}\mathcal{E}_{7,13} = 0,
  \end{align}
\end{thm}

\vspace{0.01in}
\begin{thm} With the above notation, the
  theta quotient identities of level $17$ hold:
  \allowdisplaybreaks{\begin{align}
    \label{eq:113}
  & -3 \mathcal{E}_{2,17} \mathcal{E}_{3,17}+2 \mathcal{E}_{7,13}
   \mathcal{E}_{3,17}+3 \mathcal{E}_{13,17}
   \mathcal{E}_{3,17}+\mathcal{E}_{2,17}
   \mathcal{E}_{7,17}+\mathcal{E}_{2,17} \mathcal{E}_{8,17}+3
   \mathcal{E}_{2,17} \mathcal{E}_{11,17} \\ & \qquad  -3
   \mathcal{E}_{5,17} \mathcal{E}_{11,17}+3 \mathcal{E}_{7,17}
 \mathcal{E}_{11,17}-\mathcal{E}_{2,17} \mathcal{E}_{13,17}-3
 \mathcal{E}_{7,17} \mathcal{E}_{13,17}-3 \mathcal{E}_{11,17}
 \mathcal{E}_{13,17}   = 0,  \nonumber \\ \nonumber \\ 
&-9 \mathcal{E}_{2,17} \mathcal{E}_{3,17}+5 \mathcal{E}_{7,17}
\mathcal{E}_{3,17}+9 \mathcal{E}_{13,17} \mathcal{E}_{3,17}-5
\mathcal{E}_{2,17} \mathcal{E}_{7,17}-5 \mathcal{E}_{2,17}
\mathcal{E}_{8,17}+9 \mathcal{E}_{1,17} \mathcal{E}_{11,17}  \\ &
\qquad \qquad \qquad \quad +9
\mathcal{E}_{2,17} \mathcal{E}_{11,17}+9 \mathcal{E}_{7,17}
\mathcal{E}_{11,17}+5 \mathcal{E}_{2,17} \mathcal{E}_{13,17}-9
\mathcal{E}_{11,17} \mathcal{E}_{13,17}  = 0, \nonumber \\    \nonumber \\ 
&-3 \mathcal{E}_{1,17} \mathcal{E}_{5,17}-3 \mathcal{E}_{7,17}
\mathcal{E}_{5,17}+3 \mathcal{E}_{1,17}
\mathcal{E}_{7,17}+\mathcal{E}_{2,17} \mathcal{E}_{7,17}+2
\mathcal{E}_{3,17} \mathcal{E}_{7,17}+\mathcal{E}_{2,17}
\mathcal{E}_{8,17}  \\ & \qquad \qquad \qquad\qquad \quad \qquad \qquad -3 \mathcal{E}_{2,17}
\mathcal{E}_{11,17}-3 \mathcal{E}_{7,17} \mathcal{E}_{11,17}+2
\mathcal{E}_{2,17} \mathcal{E}_{13,17}  = 0 \nonumber \\    \nonumber \\ 
&-9 \mathcal{E}_{2,17} \mathcal{E}_{3,17}+7 \mathcal{E}_{7,17}
\mathcal{E}_{3,17}+9 \mathcal{E}_{8,17} \mathcal{E}_{3,17}+2
\mathcal{E}_{2,17} \mathcal{E}_{7,17}+2 \mathcal{E}_{2,17}
\mathcal{E}_{8,17}+9 \mathcal{E}_{7,17} \mathcal{E}_{11,17}  \\ &
\qquad \qquad \qquad \qquad \qquad\qquad \quad \qquad \qquad\qquad +7 \mathcal{E}_{2,17}
\mathcal{E}_{13,17}-9 \mathcal{E}_{7,17} \mathcal{E}_{13,17}  = 0,
\nonumber \\ \nonumber \\ 
&16 \mathcal{E}_{2,17} \mathcal{E}_{7,17}-16 \mathcal{E}_{3,17}
\mathcal{E}_{7,17}+9 \mathcal{E}_{8,17} \mathcal{E}_{7,17}+9
\mathcal{E}_{1,17} \mathcal{E}_{8,17}+16 \mathcal{E}_{2,17}
\mathcal{E}_{8,17} \\ & \qquad \qquad \qquad\qquad \quad \qquad \qquad
\qquad \qquad \qquad   -9 \mathcal{E}_{2,17} \mathcal{E}_{11,17}-7
\mathcal{E}_{2,17} \mathcal{E}_{13,17} = 0, \nonumber  \\   \nonumber \\ 
&\mathcal{E}_{2,17}
\mathcal{E}_{7,17}-3 \mathcal{E}_{3,17} \mathcal{E}_{5,17}+2 \mathcal{E}_{3,17}
\mathcal{E}_{7,17}+\mathcal{E}_{2,17}
\mathcal{E}_{8,17}-\mathcal{E}_{2,17} \mathcal{E}_{13,17}-3
\mathcal{E}_{7,17} \mathcal{E}_{13,17} = 0.
  \end{align}}
\end{thm}

\vspace{0.1in}
\begin{thm} With the above notation, the theta quotient identities of
  level $19$ hold:
 {\allowdisplaybreaks  \begin{align}
    \label{eq:114}
  &  \mathcal{E}_{2,19} \mathcal{E}_{5,19}-\mathcal{E}_{4,19}
    \mathcal{E}_{5,19}+\mathcal{E}_{7,19} \mathcal{E}_{5,19} -3
    \mathcal{E}_{9,19} \mathcal{E}_{5,19}+\mathcal{E}_{11,19}
    \mathcal{E}_{5,19}+\mathcal{E}_{13,19} \mathcal{E}_{5,19}\\ &+\mathcal{E}_{7,19}
    \mathcal{E}_{9,19}-\mathcal{E}_{7,19}
    \mathcal{E}_{11,19}-\mathcal{E}_{2,19}
    \mathcal{E}_{13,19}-\mathcal{E}_{3,19}
    \mathcal{E}_{13,19}+\mathcal{E}_{9,19}
    \mathcal{E}_{13,19}-\mathcal{E}_{11,19} \mathcal{E}_{13,19} = 0, \nonumber
    \\ \nonumber \\ 
&-\mathcal{E}_{1,19} \mathcal{E}_{5,19}+\mathcal{E}_{2,19}
\mathcal{E}_{5,19}+\mathcal{E}_{3,19}
\mathcal{E}_{5,19}-\mathcal{E}_{4,19}
\mathcal{E}_{5,19}+\mathcal{E}_{7,19}
\mathcal{E}_{5,19} -\mathcal{E}_{11,19}
\mathcal{E}_{5,19} \\ & \qquad \quad\  -\mathcal{E}_{13,19} \mathcal{E}_{5,19} +\mathcal{E}_{1,19}
\mathcal{E}_{7,19}-\mathcal{E}_{2,19}
\mathcal{E}_{13,19}-\mathcal{E}_{3,19}
\mathcal{E}_{13,19}+\mathcal{E}_{11,19} \mathcal{E}_{13,19}  = 0,
\nonumber \\  \nonumber \\ 
&-\mathcal{E}_{2,19} \mathcal{E}_{5,19}-\mathcal{E}_{3,19}
\mathcal{E}_{5,19}+\mathcal{E}_{7,19}
\mathcal{E}_{5,19}-\mathcal{E}_{9,19}
\mathcal{E}_{5,19}+\mathcal{E}_{11,19}
\mathcal{E}_{5,19}\\ &\qquad \quad \ +\mathcal{E}_{13,19} \mathcal{E}_{5,19}+\mathcal{E}_{2,19} \mathcal{E}_{7,19}-\mathcal{E}_{7,19}
\mathcal{E}_{11,19}+\mathcal{E}_{2,19}
\mathcal{E}_{13,19}-\mathcal{E}_{11,19} \mathcal{E}_{13,19} = 0,
\nonumber \\ \nonumber \\ 
&\mathcal{E}_{2,19} \mathcal{E}_{5,19}+\mathcal{E}_{3,19}
\mathcal{E}_{5,19}-\mathcal{E}_{7,19}
\mathcal{E}_{5,19}-\mathcal{E}_{11,19}
\mathcal{E}_{5,19}-\mathcal{E}_{13,19}
\mathcal{E}_{5,19}+\mathcal{E}_{3,19} \mathcal{E}_{11,19} \\ &\qquad
\qquad \qquad \qquad \qquad \qquad \quad \ -\mathcal{E}_{2,19}
\mathcal{E}_{13,19}-\mathcal{E}_{3,19}
\mathcal{E}_{13,19}+\mathcal{E}_{11,19} \mathcal{E}_{13,19} = 0,
\nonumber \\ \nonumber \\ 
&-\mathcal{E}_{2,19} \mathcal{E}_{5,19}-\mathcal{E}_{3,19} \mathcal{E}_{5,19}+\mathcal{E}_{7,19} \mathcal{E}_{5,19}+\mathcal{E}_{9,19}
\mathcal{E}_{5,19} \\ & \qquad
\qquad \qquad \qquad  +\mathcal{E}_{13,19}
\mathcal{E}_{5,19}+\mathcal{E}_{2,19}
\mathcal{E}_{11,19}+\mathcal{E}_{3,19}
\mathcal{E}_{13,19}-\mathcal{E}_{11,19} \mathcal{E}_{13,19} = 0,
\nonumber \\ \nonumber \\ 
& -\mathcal{E}_{2,19}
\mathcal{E}_{5,19}-\mathcal{E}_{3,19}
\mathcal{E}_{5,19}+\mathcal{E}_{4,19}
\mathcal{E}_{5,19}+\mathcal{E}_{9,19}
\mathcal{E}_{5,19}+\mathcal{E}_{1,19}
\mathcal{E}_{9,19} \\ & \qquad
\qquad \qquad \qquad \qquad \qquad \qquad \quad \qquad \qquad +\mathcal{E}_{2,19}
\mathcal{E}_{13,19}+\mathcal{E}_{3,19} \mathcal{E}_{13,19} = 0,
\nonumber \\ \nonumber \\ 
& -\mathcal{E}_{2,19}
\mathcal{E}_{3,19}+\mathcal{E}_{5,19}
\mathcal{E}_{3,19}+\mathcal{E}_{9,19}
\mathcal{E}_{3,19}+\mathcal{E}_{2,19}
\mathcal{E}_{5,19}-\mathcal{E}_{4,19}
\mathcal{E}_{5,19}-\mathcal{E}_{2,19} \mathcal{E}_{13,19} = 0, \\ \nonumber \\ 
& \qquad \quad\ \ \  -\mathcal{E}_{3,19} \mathcal{E}_{5,19}-\mathcal{E}_{9,19}
\mathcal{E}_{5,19}+\mathcal{E}_{13,19}
\mathcal{E}_{5,19}+\mathcal{E}_{1,19}
\mathcal{E}_{13,19}+\mathcal{E}_{3,19} \mathcal{E}_{13,19} = 0, \\ \nonumber \\ 
&\qquad \qquad \qquad \qquad \qquad \   \mathcal{E}_{3,19} \mathcal{E}_{5,19}-\mathcal{E}_{4,19} \mathcal{E}_{5,19}+\mathcal{E}_{4,19}
\mathcal{E}_{11,19}-\mathcal{E}_{3,19} \mathcal{E}_{13,19} = 0, \\ \nonumber \\ 
&\qquad \qquad \qquad \qquad \qquad \qquad \qquad \quad\ \  \mathcal{E}_{2,19} \mathcal{E}_{4,19}-\mathcal{E}_{5,19} \mathcal{E}_{9,19}-\mathcal{E}_{3,19} \mathcal{E}_{13,19} = 0.
  \end{align}}
\end{thm}

For each prime level $p\ge 23$,
 a similar permuted basis construction exists for the Eisenstein subspace of
 weight one forms
 realized as the orbit of $\mathcal{E}_{p}$ under $\langle \cdot \rangle$, namely
$$\Bigl \{\langle a \rangle (\mathcal{E}_{p}) \mid a \in (\Bbb Z / p \Bbb
Z)^{*}/\{\pm 1\}  \Bigr \}.$$ These generators do not appear to be expressible
as theta quotients, as is the case in the lower levels, and the generated
graded algebra, in general, does not span
$\mathcal{M}(\Gamma_{1}(p))$. The situation may arise because $23$ is the first prime
level admitting nonzero weight one cusp forms
\cite[Corollary 3]{buzzard}. Since the space of weight one cusp forms
for $\Gamma_{1}(23)$ is spanned by $\eta(\tau) \eta(23\tau)$, a form
for $\Gamma_{0}(23)$ of multiplier $\left ( \frac{\cdot}{23} \right )$, we may construct
a corresponding $\Gamma_{0}(23)$-permuted basis of weight one forms for $ \mathcal{M}_{1}(\Gamma_{1}(23))$
\begin{align}
  \label{eq:121}
  \mathcal{M}_{1}(\Gamma_{1}(23)) = \Bbb C \eta(\tau) \eta(23\tau) \bigoplus_{a \in (\Bbb Z / 23 \Bbb
    Z)^{*}/\{\pm 1\}} \Bbb C \langle a \rangle \mathcal{E}_{23}(\tau)
\end{align}
Although our calculations suggest that product representations do not exist for
$\mathcal{E}_{a,23}$, we nevertheless conjecture that the basis
of weight one forms from \eqref{eq:121} generate the graded
algebra $\mathcal{M}(\Gamma_{1}(23))$.
Imitating the construction at higher
levels to obtain a similar $\Gamma_{0}(p)$-permuted basis of weight one forms for
$\mathcal{M}(\Gamma_{1}(p)$ is
made difficult by the lack of formulas for the dimensions of the
spaces of cusp forms of weight one for $p > 23$.

\end{document}